\documentclass[11pt,oneside]{amsart}

\usepackage{amssymb,amsmath,amsthm}
\usepackage{hyperref}
\usepackage{tikz}
\usepackage{enumitem}
\usepackage[margin = 25mm]{geometry}

\setlist[enumerate,1]{label=(\arabic*)}

\setlist[enumerate,2]{label=(\theenumi.\arabic*), ref=\theenumi.\arabic*}

\theoremstyle{definition}
\newtheorem{theorem}{Theorem}[section]
\newtheorem{remark}[theorem]{Remark}
\newtheorem{lemma}[theorem]{Lemma}
\newtheorem{definition}[theorem]{Definition}
\newtheorem{corollary}[theorem]{Corollary}

\newcommand{\TR}{\mathsf{TR}}
\newcommand{\Rule}{\mathrm{rule}}
\newcommand{\FV}{\mathrm{FV}}
\newcommand{\V}{\mathsf{V}}
\newcommand{\U}{\mathsf{U}}

\newcommand{\RCAo}{\mathsf{RCA}_0}

\newcommand{\ACAo}{\mathsf{ACA}_0}
\newcommand{\CAo}{\mathsf{CA}_0}
\newcommand{\CA}{\mathsf{CA}}
\newcommand{\HJ}{\mathrm{HJ}}

\newcommand{\hyp}{\mathchar`-}

\newcommand{\ATRo}{\mathsf{ATR}_0}
\newcommand{\ATR}{\mathsf{ATR}}
\newcommand{\TI}{\mathsf{TI}}

\newcommand{\RFN}{\mathsf{RFN}}

\newcommand{\DC}{\mathsf{DC}}
\newcommand{\SDC}{\mathsf{SDC}}
\newcommand{\M}{\mathcal{M}}
\newcommand{\K}{\mathcal{K}}
\renewcommand{\L}{\mathcal{L}}
\newcommand{\N}{\mathbb{N}}

\newcommand{\WO}{\mathrm{WO}}
\newcommand{\NN}{\mathcal{N}}
\newcommand{\ACA}{\mathsf{ACA}}

\newcommand{\GN}[1]{\ulcorner #1 \urcorner}

\DeclareMathOperator{\Con}{\mathsf{Con}}

\DeclareMathOperator{\Th}{\mathrm{Th}}

\newcommand{\Hier}{\mathrm{Hier}}
\newcommand{\B}{\mathsf{B}}
\newcommand{\ON}{\mathrm{ON}}
\newcommand{\TDC}{\mathsf{TDC}}

\newcommand{\PW}{\mathrm{PW}}
\newcommand{\TRFN}{\mathsf{TRFN}}

\newcommand{\A}{\mathsf{A}}

\title{Dependent choice and iterated $\beta_n$-model reflection}
\author{YUDAI SUZUKI}
\address{National~Institute~of~Technology, Oyama~College, Nakakuki~771, Oyama, Tochigi, Japan}
\email{yudai.suzuki.research@gmail.com}

\begin{document}

\begin{abstract}
In this paper, we consider conservation results for iterated $\beta_n$-model reflection principles.
Using these principles, 
we characterize the $\Pi^1_{m+2}$-, $\Sigma^1_{m+2}$- and $\B(\Pi^1_{m+2})$-consequences of $\Sigma^1_{n+1}\hyp\DC_0$ and $\Sigma^1_{n+1}\hyp\TDC_0$ for $m \leq n$. Here, $\B(\Pi^1_{m+2})$
denotes the set of Boolean combinations of $\Pi^1_{m+2}$-formulas, and
$\TDC$ stands for transfinite $\DC$.
We also show that,
in each case, the stronger theory has a super-exponential speed-up over the corresponding conservative subtheory.
\end{abstract}

\maketitle

\section{Introduction}
Reflection principles play a central role in the study of formal systems of arithmetic.

In first-order arithmetic, it is known that induction principles are characterized by reflection principles. Together with general properties of reflection principles, these characterizations allow us to analyze subsystems of Peano arithmetic $\mathsf{PA}$.
For the details, see, for example \cite{beklemishev2005reflection}.
In \cite{Emanuele}, Frittaion extended this result to second-order arithmetic.

Another form of reflection in second-order arithmetic is
coded $\omega$-model reflection.
Here, coded $\omega$-model reflection is a principle of the form `a formula that is true in the ground model is realized in a coded model'.
In particular, dependent choice ($\DC$) principles have this kind of characterization. Studying $\DC$ is useful for studying comprehension axioms, which are central to second-order arithmetic.
Indeed, $\Pi^1_{n+1}$-comprehension and strong $\DC$ for $\Sigma^1_{n+1}$-formulas have the same $\Pi^1_{4}$-consequences for $n  \in \omega$.

As in first-order arithmetic, there are studies of partial conservativity for $\omega$-model reflection principles.
In \cite{Cantini}, Cantini studied conservativity for ordinary $\DC$. 
The proof given in \cite{Cantini} may be seen as a hybrid of model-theoretic and proof-theoretic techniques. His proof combines partial cut elimination with asymmetric interpretation.
In \cite{RuedeThesis}, R\"{u}ede studied conservativity for transfinite $\DC$ using a method similar to Cantini's.
In \cite{ASY}, Aguilera, the author and Yokoyama studied partial conservativity for strong $\DC$ using a purely model-theoretic method. They studied not only uniform reflections but also local reflections.
We note that the work in \cite{ASY} builds on earlier work by the author
and Yokoyama \cite{suzuki2025pi}
(see also \cite{suzuki2025studies}).

In this paper, we extend the conservation results in \cite{Cantini} and \cite{RuedeThesis}.
In Section 2,
using the method of Beklemishev and Pakhomov \cite{BP}, we give a purely proof-theoretic proof that reflection principles are conservative over the corresponding reflection rules.
This includes the results in \cite{Cantini} and \cite{RuedeThesis}. We also show that local reflection principles are useful for studying the $\B(\Pi^1_{n})$-sentences provable from a reflection principle. Here, $\B(\Gamma)$ denotes the set of Boolean combinations of formulas in $\Gamma$.
In Section 3, we give another characterization of Rüede's transfinite dependent choice in terms of reflection principles.
In Section 4, we give a purely model-theoretic proof of the main conservation results in Sections 2 and 3.
We note that all conservation results given in the present paper have super-exponential speed-up.
We will discuss the details in Section 5.

\subsection*{Acknowledgements}
The author would like to thank Keita Yokoyama for his helpful comments on this paper.
The work of the author is partially supported by JSPS KAKENHI grant number JP26K17028.

\subsection*{AI Disclosure}
In preparing this paper, the author used ChatGPT 6 Astra and ChatGPT 6.1 Sol for assistance with mathematical discussions, literature searches, and the wording of the manuscript. The author independently verified all content developed with their assistance.

\section{Beklemishev-Pakhomov conservation and Cantini's result}
Throughout this paper, we write $T \equiv T'$ for theories $T$ and $T'$ if they prove the same sentences.

In \cite{Cantini}, Cantini proved the following conservation results.
\begin{itemize}
  \item $\Th_{\Pi^1_2}(\Sigma^1_1\hyp\DC_0) \equiv (\Pi^0_1\hyp\CA_0)_{<\omega^\omega}$.
  \item $\Th_{\Pi^1_3}(\Sigma^1_2\hyp\DC_0) \equiv (\Pi^1_1\hyp\CA_0)_{<\omega^\omega}$.
  \item $\Sigma^1_{n+3}\hyp\DC_0$ is $\Pi^1_4$-conservative over $(\Pi^1_{n+2}\hyp\CA_0)_{<\omega^\omega}$.
\end{itemize}
Reading his proof carefully, we can rephrase his result in terms of $\beta_n$-reflection principles.
We first present this reformulation.

\begin{definition}[{$\beta_n$-model, $\beta_n$-model reflection}]
  A (coded) $\beta_0$-model is a (coded) $\omega$-model of $\ACAo$.
  For $n > 0$, (coded) $\beta_n$-models are defined as usual. In this paper, we omit the word \textit{coded} if it is clear from the context.
  
For each $n \in \omega$, the condition [$\M$ is a (coded) $\beta_n$-model] is $\Pi^1_{n}$. We fix a $\Pi^1_n$-formula $\beta_n(\V)$ meaning that $\V$ is a $\beta_n$-model.
For a formula $\sigma(\V)$, we define the $\beta_n$-model reflection of $\sigma$ as follows.
  \begin{align*}
    \beta_n\RFN(\sigma(\V)) \equiv \forall \U \exists \M(\U,\V \in \M \land \beta_n(\M) \land \M \models \sigma(\V)).
  \end{align*}
  For $r >0$, we define the $r$-fold iterated $\beta_n$-model reflection of $\sigma$ as follows.
  \begin{align*}
    \beta_n\RFN^1(\sigma) \equiv \beta_n\RFN(\sigma),\quad \beta_n\RFN^{r+1}(\sigma) \equiv \beta_n\RFN(\sigma + \beta_n\RFN^{r}(\sigma)).
  \end{align*}
  Finally, we define $\beta_n\RFN^{\omega}(\sigma)$ as the schema $\{\beta_n\RFN^r(\sigma) : r > 0\}$.
\end{definition}
\begin{remark}
  We give a few comments on $\beta_n\RFN$.
  \begin{itemize}
    \item If $\sigma$ is a sentence, then so is $\beta_n\RFN(\sigma)$.
    \item Let $0 < r' < r$. Then, 
    $\ACAo \vdash \beta_n\RFN^{r}(\sigma) \to \beta_n\RFN^{r'}(\sigma)$. This implication is strict when $\ACAo + \beta_n\RFN^r(\sigma)$ is consistent
    because $\ACAo + \beta_n\RFN^{r}(\sigma) \vdash \Con(\beta_n\RFN^{r'}(\sigma))$.
    \item For $n > 0$, we have
$\ACAo + \beta_n\RFN(\ACAo)
  \equiv \Sigma^1_n\hyp\SDC_0$,
where $\SDC$ stands for strong dependent choice
(see \cite{Simpson,ASY}).
\end{itemize}
\end{remark}

From Cantini's proof, we obtain the following conservation results.
We note that Cantini's original formulation follows easily from our formulation.
\begin{theorem}\label{Thm Cantini rephrase}
  $\Th_{\Pi^1_{n+2}}(\Sigma^1_{n+1}\hyp\DC_0) \equiv \ACAo + \beta_n\RFN^{\omega}(\ACAo)$.
\end{theorem}
\begin{remark}
  This formulation is more informative than Cantini's original formulation in two senses.
  
  Firstly, from this formulation, we obtain $\Th_{\Pi^1_{n+2}}(\Sigma^1_{n+1}\hyp\DC_0)$ rather than merely a $\Pi^1_4$-conservative subtheory of $\Sigma^1_{n+1}\hyp\DC_0$ for $n > 2$.

  Secondly, this formulation gives a conservation result relating reflection axioms and reflection rules. It is well-known that 
  $\Sigma^1_{n+1}\hyp\DC_0$ is equivalent to the $\beta_n$-model reflection principle for $\Pi^1_{n+2}$-formulas.
  On the other hand, we can show that $\ACAo + \beta_n\RFN^{\omega}(\ACAo)$ proves the same sentences as $\ACAo$ under the $\beta_n$-model reflection rule for $\Pi^1_{n+2}$-formulas. 
  Therefore, this formulation may be seen as a conservation between them.
  We will see the details later.
\end{remark}

In this section, we give a short proof of Theorem~\ref{Thm Cantini rephrase} based on a proof given by Beklemishev and Pakhomov.
In \cite{BP}, they gave a conservation theorem for reflection axioms and reflection rules based on provability predicates and truth predicates.
We will prove that their proof can be applied to more general reflections including $\beta_n$-model reflection. As a consequence, we obtain a new proof of Cantini's conservation.

\begin{definition}
Let $R$ be an operator on $L_2$-formulas. 
We say $R$ is a generalized reflection operator if for any formula $\sigma(\vec{\V})$
\begin{itemize}
  \item $\FV(R\sigma(\vec{\V})) = \{\vec{\V}\}$,
  \item $R(\sigma(\vec{\U})) = (R\sigma)(\vec{\U})$.
\end{itemize}
In what follows, we sometimes write $\sigma(\V)$ instead of $\sigma(\vec{\V})$ for simplicity.

Let $\Gamma$ be a class of formulas. 
For a generalized reflection operator $R$, we define $\Gamma\hyp\RFN^R$, the $R$-reflection schema for $\Gamma$-formulas, as $\{\forall \V(\sigma(\V ) \to R(\sigma(\V))) : \sigma \in \Gamma\}$. We also define the $R$-reflection rule for $\Gamma$-formulas $\Gamma\hyp\RFN^{R}\hyp\Rule$ as $\displaystyle\frac{\sigma(\V)}{R(\sigma(\V))}$ for $\sigma \in \Gamma$.
\end{definition}

\begin{theorem}[{Essentially Beklemishev-Pakhomov \cite[Theorem 2]{BP}}]\label{BPconservation}

  Let $R$ be a generalized reflection operator, $\Gamma = \Pi^1_{n}$ for some $n > 0$, and let $T \subseteq \Pi^1_{n+1}$ be a theory including $\ACAo$.
  Assume that 
  \begin{enumerate}
    \item $T \vdash R(\sigma \lor \tau) \to \sigma \lor R(\tau)$ for $\sigma \in \Gamma$,
    \item for any $\sigma$, there exists $\varphi_{\sigma} \in \Gamma$ such that $T \vdash \lnot(R\sigma(\V_0)) \leftrightarrow \exists \V_1 \varphi_{\sigma}(\V_0,\V_1)$.
  \end{enumerate}
  Then, $T + \Gamma\hyp\RFN^R$ is $\Gamma$-conservative over $T + \Gamma\hyp\RFN^R\hyp\Rule$.
\end{theorem} 
\begin{proof}

  Let $\gamma \in \Gamma$ be such that 
  $T + \Gamma\hyp\RFN^R \vdash \gamma$. Then, there exist $\sigma_0,\ldots,\sigma_k \in \Gamma$ such that 
  $$T\vdash  \bigl( \forall \V(\sigma_0(\V) \to R(\sigma_0(\V))) \land \cdots \land  \forall \V(\sigma_k(\V) \to R(\sigma_k(\V))) \bigr) \to \gamma. $$
  Thus, in terms of Tait calculus, we have
  $$
  T \vdash \exists \V_0 \exists \V_1(\sigma_0(\V_0) \land \varphi_{\sigma_0}(\V_0,\V_1)),\ldots,\exists \V_0 \exists \V_1(\sigma_k(\V_0) \land \varphi_{\sigma_k}(\V_0,\V_1)), \gamma.
  $$
In particular, there are $\tau_0,\ldots,\tau_m \in T$ such that 
  $$
  \tau_0,\ldots,\tau_m \vdash \exists \V_0 \exists \V_1(\sigma_0(\V_0) \land \varphi_{\sigma_0}(\V_0,\V_1)),\ldots,\exists \V_0 \exists \V_1(\sigma_k(\V_0) \land \varphi_{\sigma_k}(\V_0,\V_1)), \gamma.
  $$
  Let $\rho \in \Pi^1_{n}$ be such that $\lnot(\tau_0 \land \cdots \land \tau_m)$ is equivalent to $\exists \V \rho$ in pure logic.
Then, 
  $$
  \vdash \exists \V \rho, \exists \V_0 \exists \V_1(\sigma_0(\V_0) \land \varphi_{\sigma_0}(\V_0,\V_1)),\ldots,\exists \V_0 \exists \V_1(\sigma_k(\V_0) \land \varphi_{\sigma_k}(\V_0,\V_1)), \gamma.
  $$
  Since this sequent is derived in pure logic, we have a cut-free proof.
  We fix such a proof tree. Then it is a sequent-labeled tree.

  We make a new labeling for this tree by removing some formulas from each node.
  The new labeling is defined by induction from the root to the leaves.
  First, we remove all formulas but $\gamma$ from the root.
  For the induction step, assume that the new labeling is defined for a node $x$.
  If $x$ is a leaf, then we do nothing. 
  Otherwise, let $x_0,\ldots,x_l$ be immediate successor nodes of $x$. That is, $x$ is inferred from those.
  We define the new labeling for those by deciding whether to delete or retain each formula occurrence.

  \vspace*{0.5em}
  [Case 1] If an occurrence is not a target of the inference, then it occurs in $x$. If this occurrence is removed from $x$, then we remove it from $x_i$. If it is kept in $x$, then we keep it in $x_i$.
  
  [Case 2] The inference between $x$ and $x_0$ is of the form $\displaystyle{\frac{\Delta,\rho(\U)}{\Delta,\exists \V \rho(\V)}}$.
  In this case, we keep $\rho(\U)$ in $x_0$.

  [Case 3] The inference between $x$ and $x_0$ is of the form $\displaystyle{\frac{\Delta,\sigma_i(\U_0) \land \varphi_{\sigma_i}(\U_0,\U_1)}{\Delta,\exists \V(\sigma_i(\U_0) \land \varphi_{\sigma_i}(\U_0,\V))}}$.
  In this case, we keep $\sigma_i(\U_0) \land \varphi_{\sigma_i}(\U_0,\U_1)$ in $x_0$.

  [Case 4] The inference is not of the form as given in Cases 2 or 3. In this case, 
  we retain each auxiliary formula if and only if we retain the corresponding principal formula. For example, assume that the inference is of the form $\displaystyle{\frac{\theta_0 \quad \theta_1}{\theta_0 \land \theta_1}}$. In this case, if $\theta_0 \land \theta_1$ is removed, then so are $\theta_0$ and $\theta_1$, and if $\theta_0 \land \theta_1$ is kept, then so are $\theta_0$ and $\theta_1$.

\vspace*{0.5em}
  Since we consider a cut-free proof, this process determines whether each formula occurrence will be removed or kept. 
  In addition, we can show that the following hold by induction.
  \begin{itemize}
    \item Every formula occurring in the new labeling is in $\Gamma$.
    \item If a formula is removed, then it is of the form $\exists \V \varphi$ for some $\varphi$.
  \end{itemize}
  
  From now on, we will show that the sequent assigned to each node by the new labeling is provable in $T + \Gamma\hyp\RFN^R\hyp\Rule$ by induction from the leaves to the root. Here, saying that a sequent $\Delta$ is provable means that the disjunction $\bigvee \Delta$ is provable.
  Since the new labeling assigns $\gamma$ to the root, this induction shows that $\gamma$ is provable in $T + \Gamma\hyp\RFN^R\hyp\Rule$. In what follows, we write $\Delta'$ for the new labeling corresponding to the original labeling $\Delta$.

Consider a leaf. Then, the original labeling for it is of the form $\Delta,P,\lnot P$ for some atomic $P$.
Then, $P,\lnot P$ are kept in the new labeling. Thus, the new labeling $(\Delta,P,\lnot P)'$ includes $P,\lnot P$. So $(\Delta,P,\lnot P)'$ is provable in $T + \Gamma\hyp\RFN^R\hyp\Rule$.

Consider a non-leaf node $x$. 
Then, there are 3 cases.

\vspace*{0.5em}
[Case 1] Assume that $x$ is a consequence of the inference $\displaystyle{\frac{\Delta,\rho(\U)}{\Delta,\exists \V \rho(\V)}}$ and $(\Delta,\rho(\U))' \equiv \Delta' \lor \rho(\U)$ is provable in $T + \Gamma\hyp\RFN^{R}\hyp\Rule$.
It is enough to show that $\Delta'$ is provable in $T + \Gamma\hyp\RFN^{R}\hyp\Rule$.
Since $\forall \V \lnot \rho(\V)$ is provable in $T$, we have 
$T+ \Gamma\hyp\RFN^{R}\hyp\Rule \vdash \lnot \rho(\U)$. Thus, $T + \Gamma\hyp\RFN^{R}\hyp\Rule \vdash \Delta'$.

[Case 2] Assume that $x$ is a consequence of $\displaystyle\frac{\Delta, \sigma_i(\U_0) \land \varphi_{\sigma_i}(\U_0,\U_1)}{\Delta, \exists \V (\sigma_i(\U_0) \land \varphi_{\sigma_i}(\U_0,\V ))}$ and $\Delta' \lor (\sigma_i(\U_0) \land \varphi_{\sigma_i}(\U_0,\U_1))$ is provable in $T + \Gamma\hyp\RFN^R\hyp\Rule$. We will prove that $\Delta'$ is also provable in $T + \Gamma\hyp\RFN^R\hyp\Rule$.
Now $T + \Gamma\hyp\RFN^R\hyp\Rule$ proves both 
$\Delta' \lor \sigma_i(\U_0)$ and $\Delta' \lor \varphi_{\sigma_i}(\U_0,\U_1)$. 
Since each formula in $\Delta'$ is in $\Gamma$ and $\sigma_i$ is in $\Gamma$, $\Delta' \lor \sigma_i(\U_0)$ is also in $\Gamma$. Thus, by $\Gamma\hyp\RFN^R\hyp\Rule$, it proves $R(\Delta' \lor \sigma_i(\U_0))$. Then, $T + \Gamma\hyp\RFN^R\hyp\Rule$ proves $\Delta' \lor R(\sigma_i(\U_0))$ by assumption (1) on $R$. Since $R(\sigma_i(\U_0)) \leftrightarrow \forall \V \lnot \varphi_{\sigma_i}(\U_0,\V)$, we have $\Delta' \lor \lnot \varphi_{\sigma_i}(\U_0,\U_1)$. Since $\Delta' \lor \varphi_{\sigma_i}(\U_0,\U_1)$ is provable, $\Delta'$ is also provable.

[Otherwise] Straightforward.
\end{proof}

We show that Theorem~\ref{BPconservation} can be applied to $\beta_n$-reflections.
\begin{lemma}\label{lem BP can be applied to beta_nRFN}
  Let $n \in \omega$. Then 
  \begin{enumerate}
    \item $\ACAo \vdash \beta_n\RFN(\sigma(\V) \lor \tau(\V)) \to \sigma(\V) \lor \beta_n\RFN(\tau(\V))$ for $\sigma \in \Pi^1_{n+2}$.
    \item For any $\sigma(\V)$, there exists $\varphi_{\sigma} \in \Pi^1_{n+2}$ such that $\ACAo \vdash \lnot \beta_n\RFN(\sigma(\V_0)) \leftrightarrow \exists \V_1\varphi_{\sigma}(\V_0,\V_1)$.
    That is, $\beta_n\RFN$ can be written as a $\Pi^1_{n+3}$-formula in $\ACAo$.
  \end{enumerate}
\end{lemma}
\begin{proof}
(1) Let $\sigma(\V) \in \Pi^1_{n+2}$. We work in $\ACAo$.
Let $X$ be a set such that $\beta_n\RFN(\sigma(X) \lor \tau(X))$ and $\lnot \sigma(X)$. 
We will show that $\beta_n\RFN(\tau(X))$ holds. Let $Y$ be a set. We will find an $\M$ such that $X,Y \in \M \land \beta_n(\M) \land \M \models \tau(X)$.

Since $\sigma \in \Pi^1_{n+2}$, $\lnot \sigma(X)$ can be written as $\exists \U \theta(X,\U)$ for some $\theta \in \Pi^1_{n+1}$. Let $Z$ be such that $\theta(X,Z)$. By $\beta_n\RFN(\sigma(X) \lor \tau(X))$, there exists $\M$ such that $X,Y,Z \in \M \land \beta_n(\M) \land \M \models (\sigma(X) \lor \tau(X))$. 
Since $\theta(X,Z)$ holds, $\M \models \theta(X,Z)$. Thus, $\M \models \lnot \sigma(X)$. Hence $\M \models \tau(X)$. This proves (1).

(2) $\beta_n\RFN$ is a $\Pi^1_{n+2}$-formula. This proves (2).
\end{proof}

\begin{definition}
  For a class $\Gamma$ of formulas, we define 
  $$\Gamma\hyp\beta_n\RFN \equiv \{\forall \V(\sigma(\V) \to \beta_n\RFN(\sigma(\V))): \sigma \in \Gamma\}.$$ We also define the inference rule $\Gamma\hyp\beta_n\RFN\hyp\Rule$ as $\displaystyle{\frac{\sigma(\V)}{\beta_n\RFN(\sigma(\V))}}$ for $\sigma \in \Gamma$. 
  As usual, we write $\Gamma\hyp\beta_n\RFN_0$ for the theory $\ACAo + \Gamma\hyp\beta_n\RFN$.
\end{definition}

We note that $\Gamma\hyp\beta_n\RFN$ (resp. $\Gamma\hyp\beta_n\RFN\hyp\Rule$) is $\Gamma\hyp\RFN^R$ (resp. $\Gamma\hyp\RFN^R\hyp\Rule$) in Theorem~\ref{BPconservation} for $R = \beta_n\RFN$.

\begin{theorem}\label{thm cons for beta_n}
$\Pi^1_{n+2}\hyp\beta_n\RFN_0$ is $\Pi^1_{n+2}$-conservative over $\ACAo + \Pi^1_{n+2}\hyp\beta_n\RFN\hyp\Rule$.
\end{theorem}
\begin{proof}
  By the previous lemma, we can apply Theorem~\ref{BPconservation} to $\beta_n\RFN$.
\end{proof}

\begin{lemma}\label{beta-ref satisfies BP}
  Let $n \in \omega$.
  \begin{enumerate}
    \item $\Sigma^1_{n+1}\hyp\DC_0$ and $\Pi^1_{n+2}\hyp\beta_n\RFN_0$ prove the same sentences.
    \item $\ACAo + \beta_n\RFN^{\omega}(\ACAo)$ and $\ACAo + \Pi^1_{n+2}\hyp\beta_n\RFN\hyp\Rule$ prove the same sentences.
  \end{enumerate}
\end{lemma}
\begin{proof}
(1) is \cite[Theorem VII.7.6 (2) and Theorem VIII.5.12]{Simpson}. We will prove (2).

We first show that the consequences of $\ACAo + \beta_n\RFN^{\omega}(\ACAo)$ are closed under the $\Pi^1_{n+2}\hyp\beta_n\RFN$-rule.
Let $\sigma(\V) \in \Pi^1_{n+2}$ be such that $\forall \V (\sigma(\V))$ is provable in  $\ACAo + \beta_n\RFN^{\omega}(\ACAo)$.
Then, for sufficiently large $r$, $\ACAo + \beta_n\RFN^{r}(\ACAo) \vdash \forall \V\sigma(\V)$.
We work in $\ACAo + \beta_n\RFN^\omega(\ACAo)$ and show that 
$\forall \V(\beta_n\RFN(\sigma(\V)))$ holds.
Let $X$ be a set. We will show that $\beta_n\RFN(\sigma(X))$ holds.
Let $Y$ be a set. Then, 
there exists a $\beta_n$-model $\M$ such that 
$X,Y \in \M \land \M \models \ACAo +\beta_n\RFN^r(\ACAo)$ by $\beta_n\RFN^{r+1}(\ACAo)$.
Then, $\M \models \forall \V(\sigma(\V))$ and hence $\M \models \sigma(X)$. Thus, we have 
$\forall Y \exists \M(X,Y \in \M \land \beta_n(\M) \land \M \models \sigma(X))$, which is $\beta_n\RFN(\sigma(X))$.

We next prove that, for all $r > 0$, $\beta_n\RFN^r(\ACAo)$ is provable in $\ACAo + \Pi^1_{n+2}\hyp\beta_n\RFN\hyp\Rule$ by induction on $r$. 
The initial case $r =1$ is trivial. 
For the induction step, notice that $\beta_n\RFN^{r+1}(\ACAo) \equiv \beta_n\RFN(\ACAo + \beta_n\RFN^r(\ACAo))$ is equivalent to $\beta_n\RFN(\beta_n\RFN^r(\ACAo))$. 
In addition, it is easy to show that $\beta_n\RFN^{r+1}(\ACAo)$ is $\Pi^1_{n+2}$. Thus, we can apply $\Pi^1_{n+2}\hyp\beta_n\RFN\hyp\Rule$ to $\beta_n\RFN^{r}(\ACAo)$ and obtain $\beta_n\RFN^{r+1}(\ACAo)$.
\end{proof}

\begin{theorem}[Theorem~\ref{Thm Cantini rephrase}, restated]
  $\Th_{\Pi^1_{n+2}}(\Sigma^1_{n+1}\hyp\DC_0) \equiv \ACAo + \beta_n\RFN^{\omega}(\ACAo)$.
\end{theorem}
\begin{proof}
  Immediate from the previous lemma and Theorem~\ref{thm cons for beta_n}.
\end{proof}

\begin{remark}
  We note that
  $\Sigma^1_{n+1}\hyp\DC_0$ has iterated exponential speed-up over $\ACAo + \beta_n\RFN^{\omega}(\ACAo)$ when each $\beta_n\RFN^r(\ACAo)$ is expressed by using the numeral for $r$.
  This is because $\Sigma^1_{n+1}\hyp\DC_0 \vdash \forall x (\beta_n\RFN^x(\ACAo) \to \beta_n\RFN^{x+1}(\ACAo))$.
\end{remark}

In \cite[Theorem 88]{RuedeThesis}, R\"{u}ede proved the following conservation result.
\begin{theorem}\label{Thm ref of DC and ref omega}
  $\ACAo + \Pi^1_2\hyp\beta_0\RFN(\Sigma^1_1\hyp\DC)$ is $\Pi^1_2$-conservative over $\ACAo + \beta_0\RFN^{\omega}(\ATR_0 + \Sigma^1_1\hyp\DC)$.
\end{theorem}
\begin{remark}
  More precisely, \cite[Theorem 88 b)]{RuedeThesis} mentioned only conservation for arithmetical formulas. However, \cite[Theorem 88 a)]{RuedeThesis} yields $\Pi^1_2$-conservation.
\end{remark}

To obtain the above result, he adapted Cantini's proof to this setting.
In what follows, we show this result  can also be obtained as a corollary of Theorem~\ref{BPconservation}.

\begin{theorem}
  Let $T$ be a finitely axiomatized theory including $\ACAo$.
  Define $\Pi^1_{n+2}\hyp\beta_n\RFN(T)$ as the schema 
  $\{\forall \V(\sigma(\V) \to \beta_n\RFN(T + \sigma(\V))) : \sigma \in \Pi^1_{n+2}\}$, and $\Pi^1_{n+2}\hyp\beta_n\RFN(T)\hyp\Rule$ as the inference rule corresponding to $\Pi^1_{n+2}\hyp\beta_n\RFN(T)$. Then $\ACAo + \Pi^1_{n+2}\hyp\beta_n\RFN(T)$ is $\Pi^1_{n+2}$-conservative over $\ACAo + \Pi^1_{n+2}\hyp\beta_n\RFN(T)\hyp\Rule$.
\end{theorem}
\begin{proof}
  By the same argument as in the proof of Lemma~\ref{lem BP can be applied to beta_nRFN}, we can apply Theorem~\ref{BPconservation} to 
  $\beta_n\RFN(T)$.
\end{proof}

As for $\Pi^1_{n+2}\hyp\beta_n\RFN\hyp\Rule$, we obtain the following results.
\begin{lemma}
  Let $T$ be a finitely axiomatized theory.
  $\ACAo + \Pi^1_{n+2}\hyp\beta_n\RFN(T)\hyp\Rule$ proves the same sentences as $\ACAo + \beta_n\RFN^{\omega}(T)$.
\end{lemma}

\begin{theorem}\label{Cor conservation for RFNT schema and RFNomega}
  Let $T$ be a finitely axiomatized theory. Then 
  $\ACAo + \Pi^1_{n+2}\hyp\beta_n\RFN(T)$ is $\Pi^1_{n+2}$-conservative over $\ACAo + \beta_n\RFN^{\omega}(T)$.
\end{theorem}

R\"{u}ede's conservation theorem follows from the previous theorem.
\begin{corollary}[Theorem~\ref{Thm ref of DC and ref omega}, restated]
  We have
  $$\Th_{\Pi^1_2}(\ACAo + \Pi^1_2\hyp\beta_0\RFN(\Sigma^1_1\hyp\DC)) \equiv \ACAo + \beta_0\RFN^{\omega}(\ATR_0 + \Sigma^1_1\hyp\DC).$$
\end{corollary}
\begin{proof}
  It is enough to show that $\Pi^1_2\hyp\beta_0\RFN(\Sigma^1_1\hyp\DC)$ is equivalent to $\Pi^1_2\hyp\beta_0\RFN(\ATRo + \Sigma^1_1\hyp\DC)$ over $\ACAo$.
  Let $\sigma(\V)$ be a $\Pi^1_2$-formula. We show that 
  $\ACAo + \Pi^1_2\hyp\beta_0\RFN(\Sigma^1_1\hyp\DC) \vdash \forall \V (\sigma(\V) \to \exists \M(\V \in \M \land \M \models \sigma(\V) + \ATRo + \Sigma^1_1\hyp\DC))$.
  
  Since $\ATRo$ is equivalent to $\beta_0\RFN(\Sigma^1_1\hyp\DC)$, $\ACAo + \Pi^1_2\hyp\beta_0\RFN(\Sigma^1_1\hyp\DC) \vdash \ATRo$.
  In addition, $\sigma(\V) \land \ATRo$ is in $\Pi^1_2$.
  Hence, we have 
  \begin{align*}
    \ACAo + \Pi^1_2\hyp\beta_0\RFN(\Sigma^1_1\hyp\DC) \vdash \sigma(\V) &\to (\sigma(\V) \land \ATRo) \\
    &\to \exists \M(\V \in \M \land \M \models \sigma(\V) + \ATRo + \Sigma^1_1\hyp\DC).
  \end{align*}
  This completes the proof.
\end{proof}

In \cite{ASY}, Aguilera, the author and Yokoyama gave $\Pi^1_{m+2},\B(\Pi^1_{m+2}),\Sigma^1_{m+2}$ subtheories of $\Sigma^1_{n+1}\hyp\SDC_0$ for $m < n$. Here, $\B(\Gamma)$ denotes the class of Boolean combinations of formulas in $\Gamma$, and $\SDC$ stands for strong $\DC$ (for strong $\DC$, see also \cite[VII.6.1.4]{Simpson}).
At the end of this section, we give similar results for ordinary $\DC$.

\begin{theorem} \label{thm lower cons for pi}
  Let $m < n$. Then,
$$\Th_{\Pi^1_{m+2}}(\Sigma^1_{n+1}\hyp\DC_0) \equiv \ACAo + \{\beta_m\RFN(\beta_n\RFN^r(\ACAo)) : r >0\}.$$
\end{theorem}
\begin{proof}
Since $\ACAo + \beta_n\RFN^{\omega}(\ACAo) \vdash \beta_m\RFN(\beta_n\RFN^r(\ACAo))$ and $\beta_m\RFN(\beta_n\RFN^r(\ACAo)) \in \Pi^1_{m+2}$, $\beta_m\RFN(\beta_n\RFN^r(\ACAo)) \in \Th_{\Pi^1_{m+2}}(\Sigma^1_{n+1}\hyp\DC_0)$.

  We show that any $\Pi^1_{m+2}$-sentence provable from $\Sigma^1_{n+1}\hyp\DC_0$ is provable in the theory $\ACAo + \{\beta_m\RFN(\beta_n\RFN^r(\ACAo)) : r >0\}$.
  Let $\theta$ be $\Pi^1_m$ such that $\Sigma^1_{n+1}\hyp\DC_0 \vdash \forall \U \exists \V \theta(\U,\V)$.
  Then, 
  $\ACAo + \beta_n\RFN^{r}(\ACAo) \vdash \forall \U \exists \V \theta(\U,\V)$ for some $r$.
  We work in $\ACAo + \beta_m\RFN(\beta_n\RFN^r(\ACAo))$ and prove that $\forall \U \exists \V \theta(\U,\V)$ holds.
  Let $X$ be a set. Take a $\beta_m$-model $\M$ such that $X \in \M \land \M \models \beta_n\RFN^r(\ACAo)$.
  Then, $\M \models \forall \U \exists \V \theta (\U,\V)$ and hence $\M \models \exists \V \theta(X,\V)$.
  Since $\exists \V \theta(X,\V)$ is $\Sigma^1_{m+1}$, it actually holds. This completes the proof.
\end{proof}

By a similar argument, we obtain the following theorem.
\begin{theorem}
  Let $m \leq n$. Let $T$ be a finitely axiomatized theory.
  Then, $\ACAo + \{\exists \M(\beta_m(\M) \land \M \models \beta_n\RFN^r(T)) : r >0\} \equiv \Th_{\Sigma^1_{m+1}}(\ACAo + \beta_n\RFN^{\omega}(T))$.
  In particular, $\ACAo + \{\exists \M(\beta_m(\M) \land \M \models \beta_n\RFN^r(\ACAo)) : r >0\} \equiv \Th_{\Sigma^1_{m+1}}(\Sigma^1_{n+1}\hyp\DC_0)$.
\end{theorem}

In \cite{ASY}, $\Th_{\B(\Pi^1_{m+2})}(\Sigma^1_{n+1}\hyp\SDC_0)$ is characterized by a kind of local reflection principle for $m \leq n$.
In what follows, we show that another kind of local reflection characterizes $\Th_{\B(\Pi^1_{n+2})}(\Sigma^1_{n+1}\hyp\DC_0)$ and $\Th_{\Sigma^1_{m+2}}(\Sigma^1_{n+1}\hyp\DC_0)$.

\begin{definition}
  Let $\Gamma$ be a class of formulas. 
  We define the sentence restriction $\Gamma^{-}$ as $\{\sigma \in \Gamma: \sigma \text{ is a sentence} \}$.
\end{definition}
\begin{theorem}\label{thm Boolean conservatioin}
  $\Th_{\B(\Pi^1_{n+2})}(\Sigma^1_{n+1}\hyp\DC_0) \equiv (\Pi^1_{n+2})^{-}\hyp\beta_n\RFN_0$.
\end{theorem}
\begin{proof}
  Note that each instance of $(\Pi^1_{n+2})^{-}\hyp\beta_n\RFN$ is a $\B(\Pi^1_{n+2})$-sentence provable in $\Sigma^1_{n+1}\hyp\DC_0$. 

  For the converse, it is enough to prove that 
  if $\sigma \in \Sigma^1_{n+2}, \pi \in \Pi^1_{n+2}$ and $\Sigma^1_{n+1}\hyp\DC_0 \vdash \sigma \lor \pi$, then $(\Pi^1_{n+2})^{-}\hyp\beta_n\RFN_0 \vdash \sigma \lor \pi$.
  Assume $\Sigma^1_{n+1}\hyp\DC_0 \vdash \sigma \lor \pi$.
  We show that $\ACAo + \beta_n\RFN^r(\lnot \sigma) \vdash \pi$ for some $r > 0$.
  This is sufficient because  $(\Pi^1_{n+2})^{-}\hyp\beta_n\RFN_0 + \lnot \sigma \vdash \beta_n\RFN^r(\lnot \sigma)$ for all $r > 0$. This is proved by induction on $r$. 
  
  By assumption, $\Sigma^1_{n+1}\hyp\DC_0 + \lnot \sigma \vdash \pi$.
  Since $\Sigma^1_{n+1}\hyp\DC_0$ is equivalent to $\Pi^1_{n+2}\hyp\beta_n\RFN_0$ and $\lnot \sigma \in \Pi^1_{n+2}$,
  $\Sigma^1_{n+1}\hyp\DC_0 + \lnot \sigma$ is equivalent to 
  $\ACAo + \Pi^1_{n+2}\hyp\beta_n\RFN(\lnot \sigma)$. Hence $\ACAo + \Pi^1_{n+2}\hyp\beta_n\RFN(\lnot \sigma) \vdash \pi$.
  By Theorem~\ref{Cor conservation for RFNT schema and RFNomega}, 
  $\ACAo + \Pi^1_{n+2}\hyp\beta_n\RFN(\lnot \sigma)$ is $\Pi^1_{n+2}$-conservative over $\ACAo + \beta_n\RFN^{\omega}(\lnot \sigma)$. Thus, $\ACAo + \beta_n\RFN^{\omega}(\lnot \sigma) \vdash \pi$.

  Now we have $\ACAo + (\Pi^1_{n+2})^{-}\hyp\beta_n\RFN + \lnot \sigma \vdash \pi$. Thus, $(\Pi^1_{n+2})^{-}\hyp\beta_n\RFN_0 \vdash \sigma \lor \pi$. This completes the proof.
\end{proof}

By arguments similar to those above, we obtain the following results.

\begin{theorem}\label{thm lower cons for Sigma}
  Let $m \leq n$.
  $$\Th_{\Sigma^1_{m+2}}(\Sigma^1_{n+1}\hyp\DC_0) \equiv \ACAo + \{\sigma \to \exists \M(\beta_m(\M) \land \M \models \beta_n\RFN^r(\sigma)) : r > 0, \sigma \in (\Pi^1_{m+2})^{-}\}.$$
\end{theorem}

\begin{theorem} \label{thm lower cons for B}
  Let $m < n$.
  $$\Th_{\B(\Pi^1_{m+2})}(\Sigma^1_{n+1}\hyp\DC_0) \equiv \ACAo + \{\sigma \to \beta_m\RFN( \beta_n\RFN^r(\sigma)) : r > 0, \sigma \in (\Pi^1_{m+2})^{-}\}.$$
\end{theorem}

\section{Transfinite Dependent Choice}
In this section, we study R\"{u}ede's transfinite dependent choice (we write $\TDC$ for short) from the perspective of $\beta_n$-model reflections.
In the first subsection, we introduce transfinite dependent choice and transfinite length reflection $\TRFN$. Then, we prove their equivalence. In addition, we give a $\Pi^1_{n+2}$-axiomatization of $\Th_{\Pi^1_{n+2}}(\TDC)$ by an argument similar to those in the previous section.
In the second subsection, we give another characterization of transfinite dependent choice in terms of transfinitely iterated reflection.
In the third subsection, we show the equivalence between our transfinitely iterated reflection and that of Pakhomov and Walsh \cite{PW}.

\subsection{Transfinite length reflection}

In \cite{Ruede, RuedeThesis}, R\"{u}ede introduced $\Sigma^1_1\hyp\TDC$, the transfinite dependent choice principle for $\Sigma^1_1$-formulas. We give the definition of $\TDC$.
\begin{definition}
  The following definitions are formalized in $\RCAo$. 
  \begin{itemize}
    \item Let $\WO(W)$ be the formula stating that $W$ is a well-order.
    \item Let $L$ be a linear order. We write $a <_L b$ to mean $a$ is smaller than $b$ in the sense of $L$.
    If $L$ is clear from the context, we merely write $a < b$.
    \item Let $W$ be a well-order. 
    \begin{itemize}
      \item We sometimes identify $W$ and its field. For example,
    we write $a \in W$ if $a$ is in the field of $W$.
      \item We write $0^W$ (or simply $0$) for the minimum element of $W$.
      \item We write $W^{+}$ for $W \setminus \{0\}$.
      \item For $a \in W$, $W\mathord{\upharpoonright}_a$ denotes the restriction of $W$ up to $a$.
      \item Let $X$ be a set. For $a \in W$, we define the $a$-th section $X_a$ of $X$ by $X_a = \{x : (a,x) \in X\}$.
      We sometimes identify $X$ and the sequence $\langle X_a \rangle_{a \in W}$. For $a \in W$, $X_{<a}$ denotes the initial segment $\langle X_b \rangle_{b < a}$.
    \end{itemize}
  \end{itemize}
\end{definition}
\begin{definition}
Define $\Gamma\hyp\TDC$ as the following schema 
\begin{align*}
  \left\{
  \begin{aligned}
    &\forall \V,W\bigl(
      \forall \U_0 \exists \U_1
      \varphi(\V,\U_0,\U_1)
      \land \WO(W) \to {}\\
    &\qquad
      \forall \U \exists \langle H_a \rangle_{a \in W}
      \bigl(\U = H_0 \land
      \forall a \in W^{+}
      \varphi(\V,H_{<a},H_a)\bigr)\bigr)
  \end{aligned}
  \;\middle|\;
  \varphi \in \Gamma
  \right\}.
\end{align*}
We write $\Gamma\hyp\TDC_0$ for $\ACAo + \Gamma\hyp\TDC$ as usual.
\end{definition}

\begin{remark}
  We give two remarks on the definition of $\TDC$.
  \begin{itemize}
    \item As for ordinary $\DC$, $\Sigma^1_{n+1}\hyp\TDC_0 \equiv \Pi^1_{n}\hyp\TDC_0$.
  In what follows, we use this equivalence implicitly.
    \item The consequent of $\TDC$, $\forall \U \exists \langle H_a \rangle_{a \in W} (\U =H_0 \land \forall a \in W^{+}\varphi(\V,H_{<a},H_a))$, can be replaced by $\forall \U \forall a \in W \exists \langle H_b \rangle_{b < a} (\U =H_0 \land \forall b \in (W\mathord{\upharpoonright}_a)^{+}\,\varphi(\V,H_{<b},H_b))$.
    \end{itemize}
\end{remark}
R\"{u}ede proved that 
$\Sigma^1_1\hyp\TDC_0$ is equivalent to $\Pi^1_2\hyp\beta_0\RFN(\Sigma^1_1\hyp\DC)$.
Therefore, we obtain the following conservation result by Theorem~\ref{Thm ref of DC and ref omega}.

\begin{theorem}
  $\Th_{\Pi^1_2}(\Sigma^1_1\hyp\TDC_0) \equiv \ACAo + \beta_0\RFN^{\omega}(\ATRo + \Sigma^1_1\hyp\DC)$.
\end{theorem}

However, for $n \geq 1$, this result cannot be generalized to the relationship between $\Sigma^1_{n+1}\hyp\TDC_0$ and $\beta_n\RFN^{\omega}(\Sigma^1_{n+1}\hyp\DC_0)$.

\begin{definition}
      Let $\theta(X,Y)$ be a formula.
    We define the formula $\Hier^\theta(W,a,A) \equiv 
    \forall b \in (W^+)\mathord{\upharpoonright}_a ( \theta(A_{<b},A_b))$.
\end{definition}
Intuitively, $\Hier^{\theta}(W,a,A)$ states that 
    the sequence $\langle A_b \rangle_{b < a}$ is a sequence obtained by transfinite dependent choice for $\theta$.

\begin{lemma}
  Let $n \in \omega$. Then, 
  $\ACAo + \Pi^1_{n+2}\hyp\TI + \Sigma^1_{n+1}\hyp\DC \vdash \Sigma^1_{n+1}\hyp\TDC$.
\end{lemma}
\begin{proof}
  We work in $\ACAo + \Pi^1_{n+2}\hyp\TI + \Sigma^1_{n+1}\hyp\DC$.
  
  Let $\theta(\U,\V)$ be a $\Pi^1_{n}$-formula such that $\forall \U \exists \V \theta(\U,\V)$. Let $W$ be a well-order.
  We show that $\forall b <_W a \forall H(\Hier^{\theta}(W,b,H) \to \exists K(H \preceq_b K \land \Hier^{\theta}(W,a,K) ))$ for all $a \in W$. Here, $H \preceq_b K$ denotes $\forall c<_W b(H_c = K_c) $. 

  Since $\forall b<_W a \forall H(\Hier^{\theta}(W,b,H) \to \exists K(H \preceq_b K \land \Hier^{\theta}(W,a,K) ))$ is $\Pi^1_{n+2}$, we can use transfinite induction to prove it. Assume that the above condition holds for all $a' < a$.

  [Case 1: $a$ is $a' + 1$ for some $a'$] This case is trivial.

  [Case 2: $a$ is a limit] Let $b < a$ and let $H$ be a set such that $\Hier^{\theta}(W,b,H)$ holds.
  Since $a$ is a limit, there is a fundamental sequence $\{a[i] : i \in \N\}$ of  $a$ such that $a[0] = b$.
  Now we have $$\forall i \forall H' \exists K ((H \preceq_{a[0]} H' \land \Hier^{\theta}(W,a[i],H')) \to (H' \preceq_{a[i]} K \land \Hier^{\theta}(W,a[i+1],K))). $$
  Since $(H \preceq_{a[0]} H' \land \Hier^{\theta}(W,a[i],H')) \to (H' \preceq_{a[i]} K \land \Hier^{\theta}(W,a[i+1],K))$ is $\Sigma^1_{n+1}$, by dependent choice, we obtain a sequence 
  $\langle K_i \rangle_{i \in \N}$ such that  
  $K_0 = H$ and 
  $$\forall i ((H \preceq_{a[0]} K_i \land \Hier^{\theta}(W,a[i],K_i)) \to (K_i \preceq_{a[i]} K_{i+1} \land \Hier^{\theta}(W,a[i+1],K_{i+1}))).$$ 
  Then, by induction on $i$, we have 
  \begin{align*}
    \forall i (H \preceq_{a[0]} K_i \preceq_{a[i]} K_{i+1} \land \Hier^{\theta}(W,a[i+1],K_{i+1})).
  \end{align*}
  Define $\widetilde{K} = \langle \widetilde{K}_c \rangle_{c < a}$ by $\widetilde{K}_c = (K_i)_c$ for the least $i$ such that $c < a[i]$. Then we have $H \preceq_{b} \widetilde{K} \land \Hier^{\theta}(W,a,\widetilde{K})$. This completes the proof.
\end{proof} 

\begin{theorem}[$\ACAo$]
  The existence of a $\beta_1$-model of $\Sigma^1_{n+1}\hyp\DC$ implies $\Con(\Sigma^1_{n+1}\hyp\TDC)$. In particular, $\beta_n\RFN^{\omega}(\Sigma^1_{n+1}\hyp\DC)$ implies $\Con(\Sigma^1_{n+1}\hyp\TDC)$ for $n \geq 1$.
\end{theorem}
\begin{proof}
  If $\M$ is a $\beta_1$-model, then $\M \models \Pi^1_{n+2}\hyp\TI$.
  Thus, a $\beta_1$-model of $\Sigma^1_{n+1}\hyp\DC$ is a model of $\Sigma^1_{n+1}\hyp\TDC$.
\end{proof}

Therefore, to obtain a characterization of $\Th_{\Pi^1_{n+2}}(\Sigma^1_{n+1}\hyp\TDC)$, we need another theory.
For this purpose, we introduce a new reflection principle.

\begin{definition}
  For a formula $\sigma(\V)$, we define the transfinite $\beta_n$-model reflection of $\sigma$ as follows.
  \begin{align*}
    \beta_n\TRFN(\sigma(\V)) \equiv \forall W(&\WO(W) \to \\ &\forall \U \exists \langle \M_a \rangle_{a \in W} \forall a \in W(\U,\V,\M_{<a} \in \M_a \land \beta_n(\M_a) \land \M_a \models \sigma(\V))).
  \end{align*}
  
  For $r > 0$, we define the $r$-fold iterated transfinite $\beta_n$-model reflection of $\sigma$ as follows.
  \begin{align*}
    \beta_n\TRFN^1(\sigma) \equiv \beta_n\TRFN(\sigma), \quad \beta_n\TRFN^{r+1}(\sigma) \equiv \beta_n\TRFN(\sigma + \beta_n\TRFN^{r}(\sigma)).
  \end{align*}
  We also define $\beta_n\TRFN^{\omega}(\sigma)$ as the schema $\{\beta_n\TRFN^r(\sigma) : r >0\}$.

   For a class $\Gamma$ of formulas, we define $\Gamma\hyp\beta_n\TRFN$ as the schema $\{\forall \V(\sigma(\V) \to \beta_n\TRFN(\sigma(\V))) : \sigma \in \Gamma\}$.
\end{definition}

We see the equivalence between transfinite $\beta_n$-reflections and transfinite recursions.
\begin{definition}
  Let $n > 0$. Let $\pi^1_n$ be a $\Pi^1_n$-universal formula.
  For a set $X$, we define a canonical $\Pi^{1,X}_n$-complete set $C_n(X)$ by $C_n(X) = \{(e,x) : \pi^1_n(e,x,X)\}$.

  For $n = 0$, we put $C_0(X) = X^{(\omega)}$. 
\end{definition}
We note that $C_1(X)$ is the same as $\HJ(X)$, the hyperjump of $X$.

The following lemma is well-known. 
For (1) -- (3), see \cite[{Lemma 2.4 (ii)}]{Cantini}. For (1), see also \cite[Lemma 3.10]{marcone2024galvin} and \cite[proof of Lemma 3.17]{suzuki2025relative}.
(4) is straightforward. For the assertion $V = L(A)$, see \cite[Chapter VII]{Simpson}.
\begin{lemma}
  The following hold.
  \begin{enumerate}
    \item For $n \leq 1$, there is a Turing functional $\Phi_n$ such that $\ACAo$ proves  $\Phi_n(C_n(X))$ is a coded $\beta_n$-model containing $X$.
    \item There is a Turing functional $\Phi_2$ such that $\Pi^1_1\hyp\CAo$ proves that $\Phi_2(C_2(X))$ is a coded $\beta_2$-model containing $X$.
    \item For $n \geq 3$, there is a Turing functional $\Phi_n$ such that $\Pi^1_{n-1}\hyp\CA_0 + \exists A(V=L(A))$ proves that $\Phi_n(C_n(X))$ is a coded $\beta_n$-model containing $X$.
    \item For $n \in \omega$, there is a Turing functional $\Psi_n$ such that 
    $$
    \ACAo \vdash \forall X,\M(X \in \M \land \beta_n(\M) \to \Psi_n((X \oplus \M)^{(n + 2)}) = C_n(X)).
    $$
  \end{enumerate}
\end{lemma}

\begin{theorem} \label{thm TRFN vs TR}
  The following hold.
  \begin{enumerate}
    \item Over $\ACAo$, $\Pi^1_n\hyp\TR$ is equivalent to $\beta_n\TRFN(\ACAo)$ for $n \leq 2$.
    \item Let $n > 2$. Then, 
    \begin{itemize}
      \item $\Pi^1_n\hyp\TR_0 + \exists A(V= L(A)) \vdash \beta_n\TRFN(\ACAo)$, 
      \item $\ACAo + \beta_n\TRFN(\ACAo) \vdash \Pi^1_n\hyp\TR_0$.
    \end{itemize}
    In particular,
    $\ACAo + \beta_n\TRFN(\ACAo)$ is $\Pi^1_4$-conservative over $\Pi^1_n\hyp\TR_0$.
  \end{enumerate}
\end{theorem}
\begin{proof}
  Immediate from the previous lemma.
\end{proof}

\begin{theorem}\label{thm tdc vs trfn}
  Let $n \in \omega$.
  Then, $\Sigma^1_{n+1}\hyp\TDC_0 \equiv \Pi^1_{n+2}\hyp\beta_n\TRFN_0$.
\end{theorem}
\begin{proof}
  We first show that $\Sigma^1_{n+1}\hyp\TDC_0$ proves $\Pi^1_{n+2}\hyp\beta_n\TRFN_0$.
  We work in $\Sigma^1_{n+1}\hyp\TDC_0$.
  Let $\theta(\U,\V)$ be a $\Pi^1_{n}$-formula such that $\forall \U \exists \V \theta(\U,\V)$ holds. We show that $\beta_n\TRFN(\forall \U \exists \V \theta)$ holds.
  Since $\Sigma^1_{n+1}\hyp\TDC_0$ includes $\Sigma^1_{n+1}\hyp\DC_0$, which is equivalent to $\Pi^1_{n+2}\hyp\beta_n\RFN_0$, 
  we obtain $\forall X \exists \M(X \in \M \land \beta_n(\M) \land \M \models \forall \U \exists \V\theta(\U,\V) )$.
  Since $(X \in \M \land \beta_n(\M) \land \M \models \forall \U \exists \V\theta(\U,\V) )$ is $\Pi^1_n$, we can apply $\Sigma^1_{n+1}\hyp\TDC$ to it. Thus, we obtain 
  \begin{align*}
  \forall X,W(&\WO(W) \to \\
  &\exists \langle \M_a \rangle_{a \in W}(X \in \M_0 \land \forall a \in W (\beta_n(\M_a) \land \M_a \models \forall \U \exists \V \theta(\U,\V)) \land  \forall a \in W^{+}(\M_{<a} \in \M_a))).
  \end{align*}
  This means that $\beta_n\TRFN(\forall \U \exists \V\theta)$ holds.

  We next show that $\Pi^1_{n+2}\hyp\beta_n\TRFN_0$ proves $\Sigma^1_{n+1}\hyp\TDC_0$.
  We work in $\Pi^1_{n+2}\hyp\beta_n\TRFN_0$.
  Let $X$ be a set, $W$ a well-order and $\theta(\U,\V)$ a $\Pi^1_{n}$-formula such that $\forall \U \exists \V \theta(\U,\V)$ holds.
  We show that there exists $H = \langle H_a \rangle_{a \in W}$ such that $X = H_0 \land \forall a \in W(\Hier^{\theta}(W,a,H))$.
  Applying $\Pi^1_{n+2}\hyp\beta_n\TRFN$ to $\forall \U \exists \V \theta(\U,\V)$, we obtain $\langle \M_a \rangle_{a \in W}$ such that 
  \begin{align*}
    \forall a \in W(X,\M_{<a} \in \M_a \land \beta_n(\M_a) \land \M_a \models \forall \U \exists \V\theta(\U,\V)).
  \end{align*}
  Then, by arithmetical transfinite recursion along $a \in W$, we can find a sequence $\langle H_a \rangle_a$ such that 
  $X = H_0 \land \forall a \in W(H_{<a}, H_a \in \M_a) \land \forall a \in W^{+} (\M_{a} \models \theta(H_{<a},H_a))$.
  Since $\M_a$ is a $\beta_n$-model for all $a \in W$, $H$ satisfies the desired conditions.
\end{proof}

We give conservation results similar to those in Section 2
based on the characterization by $\TRFN$ for $\Sigma^1_{n+1}\hyp\TDC_0$.
We first show that we can apply Theorem~\ref{BPconservation} to the $\TRFN$ operator.
\begin{lemma}
  Let $n \in \omega$. Then,
  \begin{enumerate}
    \item $\ACAo \vdash \beta_n\TRFN(\sigma \lor \tau) \to \sigma \lor \beta_n\TRFN(\tau)$ for $\sigma \in \Pi^1_{n+2}$.
    \item For any $\sigma$, $\beta_n\TRFN(\sigma)$ is $\Pi^1_{n+2}$ in $\ACAo$.
  \end{enumerate}
\end{lemma}
\begin{proof}
  The proof is essentially the same as that of Lemma~\ref{lem BP can be applied to beta_nRFN}.
\end{proof}

\begin{theorem}
  $\Sigma^1_{n+1}\hyp\TDC_0$ is $\Pi^1_{n+2}$-conservative over $\ACAo + \Pi^1_{n+2}\hyp\beta_n\TRFN\hyp\Rule$.
\end{theorem}
\begin{proof}
  This is an immediate consequence of Theorem~\ref{BPconservation} and the previous lemma.
\end{proof}

\begin{corollary}
  $\Th_{\Pi^1_{n+2}}(\Sigma^1_{n+1}\hyp\TDC_0) \equiv \ACAo + \beta_n\TRFN^{\omega}(\ACAo) $.
\end{corollary}
\begin{proof}
  It is enough to show that $\ACAo + \Pi^1_{n+2}\hyp\beta_n\TRFN\hyp\Rule$ and $\ACAo + \beta_n\TRFN^{\omega}(\ACAo)$ prove the same sentences. This is essentially the same as Lemma~\ref{beta-ref satisfies BP} (2).
\end{proof}

By the same argument as in the proof of Theorem~\ref{thm lower cons for pi}, we obtain the following result from the previous theorem.
\begin{corollary}
  Let $m < n$. Then,
    $$\Th_{\Pi^1_{m+2}}(\Sigma^1_{n+1}\hyp\TDC_0) \equiv \ACAo + \{\beta_m\RFN(\beta_n\TRFN^{r}(\ACAo)) : r > 0\}.$$
\end{corollary}

By the same argument as in the proof of Theorem~\ref{thm Boolean conservatioin}, we obtain the following result.
\begin{theorem}
  $\Th_{\B(\Pi^1_{n+2})}(\Sigma^1_{n+1}\hyp\TDC_0) \equiv \ACAo + (\Pi^1_{n+2})^{-}\hyp\beta_n\TRFN$.
\end{theorem}
Thus, we also obtain the following result by the same argument as in the proofs of Theorems~\ref{thm lower cons for Sigma} and \ref{thm lower cons for B}.
\begin{theorem}
  We have 
  \begin{itemize}
  \item $\Th_{\B(\Pi^1_{m+2})}(\Sigma^1_{n+1}\hyp\TDC_0) \equiv \ACAo + \{\sigma \to \beta_m\RFN(\beta_n\TRFN^r(\sigma)) : \sigma \in (\Pi^1_{m+2})^{-}, r > 0\}$ for $m < n$.
  \item $\Th_{\Sigma^1_{m+2}}(\Sigma^1_{n+1}\hyp\TDC_0) \equiv \ACAo + \{\sigma \to \exists \M(\beta_m(\M) \land \M \models \beta_n\TRFN^r(\sigma)) : \sigma \in (\Pi^1_{m+2})^{-}, r > 0\}$ for $m \leq n$.
  \end{itemize}
\end{theorem}

\subsection{Transfinitely iterated reflection}

By Theorem~\ref{thm tdc vs trfn}, transfinite dependent choice is equivalent to transfinite length reflection. 
Ordinary dependent choice, which makes a sequence of length $\omega$, corresponds to $\beta_n$-model reflection: dependent choice of length $\omega$ yields a $\beta_n$-model, and vice versa.
This correspondence suggests that dependent choice of length $\omega^2$ corresponds to $2$-fold $\beta_n$-model reflection. 
Moreover, iterating this argument, we may obtain the equivalence of $\omega^{\alpha}$-length dependent choice and $\alpha$-fold iterated $\beta_n$-model reflection.
In this subsection, we justify this argument.

We begin by introducing transfinitely iterated $\beta_n$-model reflection.
In \cite{PW}, Pakhomov and Walsh defined transfinitely iterated $\beta_0$-model reflection over $\ACAo$ based on G\"{o}del's fixed-point lemma. The idea of this formulation goes back to Feferman.

In this paper, we give a definition based directly on transfinite recursion, working over $\mathsf{ATR}_0$.
In the next subsection, we show the equivalence between the definition in \cite{PW} and our definition over $\ATRo$.
\begin{definition}[$\ATRo$]
  Let $\sigma(\V)$ be a formula, $W$ a well-order, 
  $\M =  \langle M_i \rangle_{i \in \N}$ a coded $\beta_0$-model.
  
  For a set $X \in \M$, define $I_X = \{i : X \in M_i \land \M \models \beta_n(M_i)\}$.
  We define a sequence $H^{\sigma(X)}(\M) = \langle H^{\sigma(X)}_a(\M) \rangle_{a \in W}$ by recursion on $a \in W$ as follows.
  \begin{align*}
    i \in H^{\sigma(X)}_a(\M) \leftrightarrow  i \in I_X   \land  \forall b <_W a \forall j(&M_j \in M_i \to \\
    &\exists k \in H^{\sigma(X)}_b(\M)(M_j \in M_k \in M_i \land M_k \models \sigma(X))).
  \end{align*}
  We write $\M \models \beta_n\RFN^{W}(\sigma(X))$ if $\forall a \in W \forall k \exists i \in H^{\sigma(X)}_a(\M) (M_k \in M_i)$.

  In what follows, we write $H^{\sigma(X)}_a$ instead of $H^{\sigma(X)}_a(\M)$ if it is clear from the context.
  We also write $\K \in H^{\sigma(X)}_a$ for $\exists i(\K = M_i \land i \in H^{\sigma(X)}_a)$. Thus, $\M \models \beta_n\RFN^W(\sigma(X))$ can be written as $\forall a \in W \forall Y \in \M \exists \K \in H^{\sigma(X)}_a(Y \in \K)$.
\end{definition}
Intuitively, $i \in H^{\sigma(X)}_a$ means that $\M$ believes that $M_i$ is a $\beta_n$-model of $a$-fold iterated $\beta_n$-reflection of $\sigma$. 

\begin{remark}
  We note that $H^{\sigma}$ is $\Delta^{1,\M \oplus W}_1$. 
  In particular, if $W \in \M$, then $\langle H_a \rangle_{a \in W}$ is $\Delta^{1,\M}_1$.
  Moreover, $\M \models \beta_n\RFN^W(\sigma(X))$ is also $\Delta^{1,\M \oplus W}_1$ because 
  \begin{align*}
    \M \models \beta_n\RFN^W(\sigma(X)) &\leftrightarrow \forall H(H = H^{\sigma(X)} \to \forall a \in W \forall k \exists i \in H_a(M_k \in M_i)) \\
    &\leftrightarrow \exists H(H = H^{\sigma(X)} \land \forall a \in W \forall k \exists i \in H_a(M_k \in M_i)).
  \end{align*}
  Thus, occurrences of $H^{\sigma}$ do not change the complexity of $\Pi^1_n$-formulas for $n \geq 1$. 
\end{remark}

\begin{lemma}[$\ATRo$]\label{lem RFNW and RFN}
  Let $\sigma(\V)$ be a formula.
  Let $W$ be a well-order having a non-minimum element, $\M$ be a $\beta_0$-model such that $\M \models \beta_n\RFN^{W}(\sigma(X))$.
  Then $\M \models \beta_n\RFN(\sigma(X))$.
\end{lemma}
\begin{proof}
  Let $1 \in W$ be the second least element in $W$.
  Take a set $Y \in \M$. Then, there exists $\K' \in \M (Y \in \K' \land \K' \in H^{\sigma(X)}_1)$. 
  Then, there exists $\K \in \K' (Y \in \K \land \K \in H^{\sigma(X)}_0 \land \K \models \sigma(X) )$.
  Thus, we obtain $\forall Y \in \M \exists \K \in \M(\M \models [Y \in \K \land \beta_n(\K) \land \K \models \sigma(X)])$. This completes the proof.
\end{proof}

\begin{lemma}[$\ATRo$]
  Let $\sigma(\V)$ be a formula, $W$ a well-order, $\M$ a coded $\beta_0$-model and $X \in \M$. Then, 
  $\forall a,b \in W(b < a \to H^{\sigma(X)}_b \supseteq H^{\sigma(X)}_a)$.
\end{lemma}
\begin{proof}
  By definition.
\end{proof}

\begin{definition}
  Let $\Gamma$ be a class of formulas.
  We define $\Gamma\hyp\beta_n\RFN^{\ON}$ as 
  \begin{align*}
   \{\forall \V(\sigma(\V) \to  \forall W,\U(\WO(W) \to \exists \M(\U,\V \in \M \land \beta_n(\M) \land \M \models \beta_n\RFN^W(\sigma(\V)) ))) : \sigma \in \Gamma\}.
  \end{align*}
\end{definition}
\begin{remark}
  We note that we can omit $\U$ from the above definition by taking $\sigma'(\V \oplus \U) \equiv \sigma(\V) \land \U = \U$.
\end{remark}

We will show that $\Pi^1_{n+2}\hyp\beta_n\RFN^{\ON}$ is equivalent to $\Sigma^1_{n+1}\hyp\TDC$.
As we mentioned at the beginning of this subsection, 
it seems that dependent choice of length $\omega^a$ corresponds to $a$-fold $\beta_n$-model reflection.
To establish this correspondence,
we introduce the standard ordinal notation for $\omega^a$.
\begin{definition}[$\RCAo$]
    Let $L$ be a linear ordering.
    We define $\omega^L$ as the following linear ordering. The field $|\omega^{L}|$ is defined as 
    $|\omega^L| = \{\langle \rangle\} \cup \{\langle (x_0,n_0),\ldots,(x_k,n_k) \rangle : x_0 >_L x_1 >_L \cdots >_L x_k \land n_0,\ldots,n_k> 0\}$.
    The empty sequence $\langle \rangle$ is the minimum element.
    For $\langle (x_0,n_0),\ldots,(x_k,n_k) \rangle$ and $\langle (y_0,m_0),\ldots,(y_l,m_l) \rangle \in |\omega^L|$,
    we write 
    $$\langle (x_0,n_0),\ldots,(x_k,n_k) \rangle <_{\omega^L} \langle (y_0,m_0),\ldots,(y_l,m_l) \rangle$$ if one of the following holds.
    \begin{itemize}
      \item $\langle (x_0,n_0),\ldots,(x_k,n_k) \rangle$ is a proper initial segment of $\langle (y_0,m_0),\ldots,(y_l,m_l) \rangle$,
      \item 
    $\exists i \leq \min\{k,l\}(\forall j < i (x_j = y_j \land n_j = m_j) \land ((x_i = y_i \land n_i < m_i) \lor (x_i <_L y_i)))$.
    \end{itemize}
\end{definition}

It is well-known that the mapping $W \mapsto \omega^W$ preserves the well-foundedness of $W$ in $\ACAo$ \cite[p.299]{Girard}. 

Intuitively, an element $\langle (x_0,n_0),\ldots,(x_k,n_k)\rangle$ denotes the ordinal $$
\omega^{x_0}\cdot n_0 + \cdots + \omega^{x_k}\cdot n_k
$$
written in the Cantor normal form. We can naturally define addition on $|\omega^L|$ based on this intuition.

\begin{definition}
  Let $L$ be a linear order, and let $\langle \K_x \rangle_{x \in L}$ be a sequence of $\beta_0$-models.
  \begin{itemize}
    \item We define the initial segment $\K_{<x}$ for $x \in L$ as $\langle \K_y \rangle_{y < x}$. (See also Definition 3.1.)
    \item We define the union $\K_L = \bigcup_{x \in L} \K_x$ as the sequence $\langle (\K_x)_i : x \in L, i \in \N \rangle$. Here, $(\K_x)_i$ is the $i$-th element of $\K_x$. $\bigcup_{y <_L x} \K_{y}$ is defined as $\K_{L\mathord{\upharpoonright_x}}$.
  \end{itemize}
\end{definition}
We note that if $\K_{<x} \in \K'$ for a $\beta_0$-model $\K'$, then $\bigcup_{y <_L x} \K_{y} \in \K'$.

\begin{lemma}[$\ACAo$] \label{Lem union of beta}
  Let $L$ be a nonempty linear order, $\langle \K_x \rangle_{x \in L}$ a sequence of $\beta_n$-models such that 
  $\forall x,y \in L(x <_L y \to \K_x \subseteq \K_y)$.
  Then, $\K_L$ is also a $\beta_n$-model.
  In addition, if $\sigma$ is a $\Pi^1_{n+2}$-formula and $X$ is a set such that $\forall x \in L (\K_x \models \sigma(X))$, then $\K_L \models \sigma(X)$.
\end{lemma}
\begin{proof}
  We prove by meta-induction on $n$.

  [The initial step] It is enough to show that $\K_L$ is closed under Turing sum, Turing reduction and Turing jump. This is trivial.

  [induction step] Let $L$ be a linear order, $\langle \K_x \rangle_{x \in L}$ a sequence of $\beta_{n+1}$-models such that 
  $\forall x,y \in L(x <_L y \to \K_x \subseteq \K_y)$.
  Then, by the induction hypothesis, $\K_L$ is a $\beta_n$-model.

  Let $\theta(\vec{\V},\U)$ be a $\Pi^1_{n}$-formula and $\vec{X} \in \K_L$ be such that $\exists Y \theta(\vec{X},Y)$ holds. We will show that $\K_L \models \exists Y \theta(\vec{X},Y)$.
  Now, for sufficiently large $x \in L$, $\vec{X} \in \K_x $.
  Since $\K_x$ is a $\beta_{n+1}$-model, $\exists Y \in \K_x (\K_x \models \theta(\vec{X},Y))$.
  For such a $Y$, $\theta(\vec{X},Y)$ actually holds because $\K_x$ is a $\beta_{n+1}$-model. Since $\vec{X},Y \in \K_L$ and $\K_L$ is a $\beta_n$-model, $\K_L \models \theta(\vec{X},Y)$.

  For the additional part, write $\sigma(\V) \equiv \forall \U_0 \exists \U_1 \theta(\V,\U_0,\U_1)$.
  Let $Y \in \K_L$. Then, $Y \in \K_x$ for some $x \in L$.
  Since $\K_x \models \sigma(X)$, there exists $Z \in \K_x (\K_x \models \theta(X,Y,Z))$. Since both $\K_x$ and $\K_L$ are $\beta_n$-models, $\K_L \models \theta(X,Y,Z)$. Thus, $\K_L \models \forall \U_0 \exists \U_1 \theta(X,\U_0,\U_1)$.
\end{proof}

\begin{lemma}[$\ATRo$] \label{lem seq of beta makes betaON}
    Let $\sigma(\V)$ be a $\Pi^1_{n+2}$-formula, $W$ a well-ordering and $X$ a set.
    Let $\langle \K_x \rangle_{x \in \omega^{W}}$ be a sequence of coded $\beta_n$-models satisfying 
    $$\forall x \in \omega^W (X,\K_{<x} \in \K_x \land \K_x \models \sigma(X)).$$
    Put $\M= \bigcup_{x \in \omega^W} \K_x$.
    Let $P(a)$ be the following condition for $a \in W$.
    \begin{align*}
        P(a) \equiv \forall x \in \omega^W (\bigcup_{y < x +\omega^a} \K_{y} \in H^{\sigma(X)}_{a}).
    \end{align*}
    Then, $\forall a \in W P(a)$. In addition, $\M$ is a $\beta_n$-model of $\beta_n\RFN^{W}(\sigma(X))$.
\end{lemma}
\begin{proof}
    We prove by induction on $a \in W$. Since $P(a)$ is $\Delta^1_1$ in $\ATRo$, this induction works.
    Assume that $P(b)$ holds for all $b < a$.
    In what follows, let $\L_x$ denote $\bigcup_{y < x} \K_y$ for each $x \in \omega^W$. We note that since $\K_{<x} \in \K_x$, $\L_{x} \in \K_x$ for all $x \in \omega^W$.

    [Case: $a = 0$]
    We note that 
    $\NN \in H^{\sigma(X)}_0$ if and only if 
    $\M \models \beta_n(\NN) \land X \in \NN$.
    Since $\M$ is a $\beta_n$-model, the latter is the same as $\beta_n(\NN) \land X \in \NN$.
    Let $x \in \omega^W$. Then, $X \in \L_{x + \omega^a}$.
    Since $\beta_n(\K_y)$ for all $y < x + \omega^a$, $\beta_n(\L_{x + \omega^a})$ by Lemma~\ref{Lem union of beta}. Thus, $\L_{x + \omega^a} \in H^{\sigma(X)}_0$.

    [Case $a > 0$]
    Let $x \in \omega^W$. For the same reason as in the case $a = 0$, $\M \models \beta_n(\L_{x + \omega^a}) \land X \in \L_{x + \omega^a} $.
Thus, to prove $\L_{x + \omega^a} \in H^{\sigma(X)}_a$, it is enough to show that 
    \begin{align*}
      \forall b < a \forall Y \in \L_{x + \omega^a} \exists \NN \in H^{\sigma(X)}_b (Y \in \NN \in \L_{x + \omega^a} \land \NN \models \sigma(X)).
    \end{align*}
    Let $b < a$ and $Y \in \L_{x + \omega^a}$.
    Take a $y < x + \omega^a$ such that $Y \in \K_y$.
We show that $\NN = \L_{y + \omega^b}$ satisfies the desired conditions.

    \begin{itemize}
      \item We prove that $Y \in \L_{y + \omega^b} \in \L_{x + \omega^a}$. Now $y < y + \omega^b < x + \omega^a$. Thus, $Y \in \L_{y + \omega^b}$ and $\L_{y + \omega^b} \in \K_{y + \omega^b}$ and $\K_{y+ \omega^b} \in \K_{y + \omega^b + 1} \subseteq  \L_{x + \omega^a}$. Hence  $\L_{y+ \omega^b} \in \L_{x + \omega^a}$. 
      \item By the induction hypothesis, $\L_{y + \omega^b} \in H^{\sigma(X)}_b$.
      \item Since $\K_{z} \models \sigma(X)$ for all $z < y + \omega^b$, $\L_{y + \omega^b} \models \sigma(X)$ by Lemma~\ref{Lem union of beta}.
    \end{itemize}
    This completes the induction.
    
Finally, we show that $\beta_n(\M)$ and $\M \models \beta_n\RFN^W(\sigma(X))$. 
$\beta_n(\M)$ follows from Lemma~\ref{Lem union of beta}.
To see $\M \models \beta_n\RFN^W(\sigma(X))$, it is enough to show that $\forall a \in W \forall Y \in \M \exists \K \in H^{\sigma(X)}_a (Y \in \K)$.
Take arbitrary $a \in W$ and $Y \in \M$. Then, $Y \in \L_{x + \omega^a}$ for some $x \in \omega^W$. $\K = \L_{x + \omega^a}$ satisfies the desired condition.
\end{proof}

\begin{theorem}
  $\Sigma^1_{n+1}\hyp\TDC_0$ proves $\Pi^1_{n+2}\hyp\beta_n\RFN^{\ON}$.
\end{theorem}
\begin{proof}
  We work in $\Sigma^1_{n+1}\hyp\TDC_0$.
  Let $\sigma(\V)$ be a $\Pi^1_{n+2}$-formula, $W$ a well-order, and $X$ a set such that $\sigma(X)$ holds.
  We prove that there exists a $\beta_n$-model $\M$ such that $X \in \M \land \M \models \beta_n\RFN^W(\sigma(X))$.
By Lemma~\ref{lem seq of beta makes betaON}, it is enough to show that there exists $\langle \K_x \rangle_{x \in \omega^W}$ such that 
\begin{align*}
  \forall x \in \omega^W(\beta_n(\K_x) \land X,\K_{<x} \in \K_x \land \K_x \models \sigma(X)).
\end{align*}
Since $\Sigma^1_{n+1}\hyp\TDC_0$ includes $\Sigma^1_{n+1}\hyp\DC_0$, which is equivalent to $\Pi^1_{n+2}\hyp\beta_n\RFN$,
we obtain $\forall \V \exists \K(\V,X \in \K \land \beta_n(\K) \land \K \models \sigma(X))$.
In particular, there exists $\L$ such that $X \in \L \land \beta_n(\L) \land \L \models \sigma(X)$.

Applying transfinite dependent choice to the formula $\V,X \in \K \land \beta_n(\K) \land \K \models \sigma(X)$, 
we obtain 
a sequence $\langle \K_x \rangle_{x \in \omega^W}$ such that $\K_0 = \L$ and $\forall x \in \omega^W(X,\K_{<x} \in \K_x \land \beta_n(\K_x) \land \K_x \models \sigma(X))$.
\end{proof}

We next show that the converse holds.

\begin{definition}
For a linear order $L$, $a,b \in L$ such that $a \leq_L b$ and sequences $X = \langle X_{a'}\rangle_{a' <_L a}, Y = \langle Y_{b'} \rangle_{b' <_L b}$,
    we write $X \preceq Y$ if $\forall a' <_ L a (X_{a'} = Y_{a'})$.
\end{definition}

\begin{lemma}[$\ATRo$] \label{lem absoluteness of H}
    Let $\sigma(\V)$ be a formula, $X$ a set and $W$ a well-order.
    Let $\K \in \K' \in \M$ be coded $\beta_n$-models such that 
    $X \in \K \land W \in \K' \land \K' \models \ATRo$.
    Then, $\forall a \in W \forall \L \in \K(\L \in (H^{\sigma(X)}_a(\K))^{\K'} \leftrightarrow \L \in H^{\sigma(X)}_a(\M))$.
\end{lemma}
\begin{proof}
    We note that $\forall \L \in \K(\beta_n(\L) \leftrightarrow \K \models \beta_n(\L) \leftrightarrow \M \models \beta_n(\L))$ because both $\K$ and $\M$ are $\beta_n$-models.
    Thus, to consider $\L \in (H^{\sigma(X)}_a(\K))^{\K'}$ or $\L \in H^{\sigma(X)}_a(\M)$, we need only consider  $\beta_n$-models $\L$ such that $X \in \L \in \K$.

    We will prove by induction on $a$.
    Assume the desired condition holds for all $b <_W a$.
    Then, for any $\L \in \K$ such that $X \in \L \land \beta_n(\L)$,
    \begin{align*}
        \K' \models \L \in H^{\sigma(X)}_a(\K) &\leftrightarrow 
        \K' \models \forall b < a \forall Y \in \L \exists \L' \in \L(Y \in \L'  \land \L' \models \sigma(X) \land \L' \in H^{\sigma(X)}_b(\K) ) \\
        & \leftrightarrow \forall b < a \forall Y \in \L \exists \L' \in \L(Y \in \L' \land \L' \models \sigma(X)  \land \K' \models \L' \in H^{\sigma(X)}_b(\K) ) \\
        & \leftrightarrow \forall b < a \forall Y \in \L \exists \L' \in \L(Y \in \L' \land \L' \models \sigma(X)  \land \L' \in H^{\sigma(X)}_b(\M) ) \\
        &\leftrightarrow \L \in H^{\sigma(X)}_a(\M).
      \end{align*}
      This completes the proof.
\end{proof}

\begin{lemma}[$\ATRo$] \label{lem H and beta}
  Let $\sigma(\V)$ be a $\Pi^1_{n+2}$-formula, $\M$ a coded $\beta_n$-model such that $X \in \M$ and $W$ a well-order.
If $\K \in H^{\sigma(X)}_a(\M)$ for some $a > 0$, then $\K \models \sigma(X)$.
\end{lemma}
\begin{proof}
  Since $a \geq 1$, $\K \in H^{\sigma(X)}_1(\M)$.
  By the definition of $H^{\sigma(X)}_1$, $\forall Y \in \K \exists \L \in \K(Y \in \L \land \L \models \sigma(X) \land \L \in H^{\sigma(X)}_0)$.
  Since both $\M$ and $\K$ are $\beta_n$-models, we obtain 
  \begin{align*}
    \L \in H^{\sigma(X)}_0 &\leftrightarrow X \in \L \land \M \models \beta_n(\L) \\
  &\leftrightarrow X \in \L \land \K \models \beta_n(\L).
  \end{align*}
  Therefore, $\K$ is a $\beta_n$-model of $\beta_n\RFN(\sigma(X))$. This implies $\K \models \sigma(X)$.
\end{proof}

Recall that $\Hier^{\theta}(W,a,A)$ is the formula stating that $A = \langle A_b \rangle_{b<_W a}$ is a sequence obtained by transfinite dependent choice for $\theta$ along $W\mathord{\upharpoonright}_a$.

\begin{lemma}
    Let $\sigma(\V) \equiv \forall \U_0 \exists \U_1 \theta (\V,\U_0,\U_1)$ be a $\Pi^1_{n+2}$-formula such that 
    $\ACAo \vdash \forall \V(\sigma(\V) \to \ATRo)$. 
    Then, the following assertion is provable in $\ATRo$.

    Let $W$ be a well-order, 
    $\M$ a coded $\beta_n$-model and $X \in \M$.
    Then, for any $a \in W$, the following holds:
    Let $\K,\K' \in \M$ be coded $\beta_n$-models of $\sigma(X)$ such that $W \in \K \in \K' \land \K \in H^{\sigma(X)}_a(\M)$.
    Then, for any $x \leq_W a$ and $A = \langle A_{y} \rangle_{y <_W x} \in \K$, if $\Hier^{\theta}(W,x,A)$ then 
    there exists $B = \langle B_{y} \rangle_{y <_W a} \in \K'$ such that $A \preceq B \land \Hier^{\theta}(W,a,B)$.
\end{lemma}
\begin{proof}
    We prove by induction on $a$. Since the desired condition for $a$ is $\Delta^{1,\M \oplus H^{\sigma(X)}(\M)}_1$, this induction works. In particular, since $\M$ is a $\beta_n$-model, $\Hier^{\theta}(W,x,A)$ can be read as $\M \models \Hier^{\theta}(W,x,A)$.

    Assume the desired condition holds for all $b <_W a$.
    Let $\K,\K' \in \M$ be coded $\beta_n$-models of $\sigma(X)$ such that $W \in \K \in \K' \land \K \in H^{\sigma(X)}_a$.
    Let $x \leq_W a$.
    Let $A = \langle A_{y} \rangle_{y <_W x} \in \K$ be such that $\Hier^{\theta}(W,x,A)$.
    If $x = a$, then $A$ itself is the desired $B$. Thus, we may assume that $x < a$.

    We note that all of $\M,\K,\K'$ are $\beta_n$-models and $\theta$ is $\Pi^1_{n}$. Thus,
    $\Hier^{\theta}$ is absolute for these models.

    [Case: $a = a' + 1$.] By the induction hypothesis, 
    we can take $B = \langle B_{y} \rangle_{y < a'} \in \K'$ such that $A \preceq B \land \Hier^{\theta}(W,a',B)$.
    Since $\K' \models \sigma(X)$, there exists $B_{a'} \in \K'$ such that $\K' \models \theta(X,B,B_{a'})$.
    Then, the concatenation $B' = B \ast \langle B_{a'} \rangle$ is also in $\K'$ and satisfies $A \preceq B' \land \Hier^{\theta}(W,a,B')$.

    [Case: $a$ is a limit.] Let $\langle a[i] \rangle_{i \in \N} \in \K$ be a fundamental sequence of $a$ such that $a[0] > x$.
    We can take such a sequence because $W \in \K$ and $\K \models \ACAo$.
    Since $\K \in H^{\sigma(X)}_a$, $\forall i \in \N \forall Y \in \K \exists \L \in \K (Y \in \L \land \L \in H^{\sigma(X)}_{a[i]}(\M))$.
    Since $\K' \models \ATRo$, we can make $H^{\sigma(X)}(\K)$ within $\K'$. By Lemma~\ref{lem absoluteness of H}, $\L \in (H^{\sigma}_a(\K))^{\K'} \leftrightarrow \L \in H^{\sigma(X)}_a(\M)$. Thus, we obtain 
    \begin{align*}
      \K' \models \forall i \in \N \forall Y \in \K \exists \L \in \K (Y \in \L \land \L \in H^{\sigma(X)}_{a[i]}(\K)).
    \end{align*}
    
    Then, we can find a sequence $\langle \K_i \rangle_{i \in \N} \in \K'$ such that 
    $W,A \in \K_0 \land \forall i(\K_i \in \K_{i+1} \in \K \land \K_i \in H^{\sigma(X)}_{a[i]})$.

    We note that $\K_i \models \sigma(X)$ for all $i$ by Lemma~\ref{lem H and beta} because $a[i] > x \geq 0$.
    Applying the induction hypothesis to $\K_0 \in \K_1$, we obtain
     $B^0 = \langle B^0_{y} \rangle_{y < a[0]} \in \K_1$ such that $A \preceq B^0 \land \Hier^{\theta}(W,a[0],B^0)$.
    Similarly, we can take $B^i \in \K_{i+1}$ for each $i \geq 0$ such that 
    $B^{i} \preceq B^{i+1} \land \Hier^{\theta}(W,a[i+1],B^{i+1})$.
    Since the sequence $\langle B^i \rangle_{i \in \N}$ can be chosen to be arithmetically definable from $\langle \K_i \rangle_{i \in \N}$, $\langle B^i \rangle_{i \in \N} \in \K'$.
    Put $C = \bigcup_{i \in \N} B^i$. Then $C \in \K' \land A \preceq C \land \Hier^{\theta}(W,a,C)$.
\end{proof}

\begin{theorem}
    Over $\ATRo$, $\Pi^1_{n+2}\hyp\beta_n\RFN^{\ON}$ implies $\Sigma^1_{n+1}\hyp\TDC$.
\end{theorem}
\begin{proof}
    We work in $\ATRo + \Pi^1_{n+2}\hyp\beta_n\RFN^{\ON}$.
    Let $W$ be a well-order,  $\theta(\V,\U_0,\U_1)$ a $\Pi^1_n$-formula and $X$ a set such that 
    $\forall \U_0 \exists \U_1 \theta(X,\U_0,\U_1)$. We may assume that $W$ has at least two elements.

    Let $Y$ be a set. We show that $\forall a \in W \exists A (\langle Y \rangle \preceq A \land \Hier^{\theta}(W,a,A))$.
    By $\Pi^1_{n+2}\hyp\beta_n\RFN^{\ON}$, there exists a coded $\beta_n$-model $\M$ such that 
    $X,Y,W \in \M \land \M \models \beta_n\RFN^W(\ATRo + \forall \U_0 \exists \U_1 \theta(X,\U_0,\U_1))$.
    Then, for any $a \in W^+$, there exists $\K \in \M$ such that $\langle Y \rangle, W \in \K \in H^{(\ATR_0 \land \forall \U_0 \exists \U_1 \theta(X))}_a(\M)$.
    In addition, there exists a $\beta_n$-model $\K' \in \M( \K \in \K' \land \K' \models \ATRo + \forall \U_0 \exists \U_1 \theta(X,\U_0,\U_1))$ because $\M \models \beta_n\RFN(\ATRo + \forall \U_0 \exists \U_1 \theta(X,\U_0,\U_1))$ by Lemma~\ref{lem RFNW and RFN}.
    Applying the previous lemma to $\K$ and $\K'$, we obtain $A$ satisfying the desired condition.
\end{proof}

\begin{corollary}
  Over $\ACAo$, $\ATRo + \Pi^1_{n+2}\hyp\beta_n\RFN^{\ON} \leftrightarrow \Sigma^1_{n+1}\hyp\TDC$.
\end{corollary}

We next show that Beklemishev-Pakhomov type conservation holds for $\beta_n\RFN^{\ON}$.

\begin{definition}
  For a formula $\sigma(\V)$, define $\beta_n\RFN^{\ON}(\sigma(\V))$ as 
  \begin{align*}
    \beta_n\RFN^{\ON}(\sigma(\V)) \equiv \forall W,\U(\WO(W) \to \exists \M(\V,\U \in \M \land \beta_n(\M) \land \M \models \beta_n\RFN^W(\sigma(\V)))).
  \end{align*}
  We also define $\Gamma\hyp\beta_n\RFN^{\ON}\hyp\Rule$ as $\displaystyle\frac{\sigma(\V)}{\beta_n\RFN^{\ON}(\sigma(\V))}$ for $\sigma \in \Gamma$.
\end{definition}
We note that $\Pi^1_{n+2}\hyp\beta_n\RFN^{\ON}$ can be written as $\{\forall \V(\sigma(\V) \to \beta_n\RFN^{\ON}(\sigma(\V))) : \sigma \in \Pi^1_{n+2}\}$.

\begin{lemma}
  Let $n \in \omega$. 
  \begin{enumerate}
    \item $\ATRo \vdash \beta_n\RFN^{\ON}(\sigma \lor \tau) \to \sigma \lor \beta_n\RFN^{\ON}(\tau)$ for $\sigma \in \Pi^1_{n+2}$.
    \item For any $\sigma(\V)$, $\beta_n\RFN^{\ON}(\sigma(\V))$ is $\Pi^1_{n+2}$ in $\ATRo$. (Hence, it is $\Pi^1_{n+3}$.)
  \end{enumerate}
\end{lemma}
\begin{proof}
(1) Let $\sigma(\V) \in \Pi^1_{n+2}$. Write $\lnot \sigma \equiv \exists \U \theta(\V,\U) $. 

We work in $\ATRo$. Let $X$ be a set such that 
$\beta_n\RFN^{\ON}(\sigma(X) \lor \tau(X)) \land \lnot \sigma(X)$. We prove that 
$\beta_n\RFN^{\ON}(\tau(X))$ holds. 

Let $Y$ be a set, $W$ a well-order. Since $\lnot \sigma(X)$, there exists $Z$ such that $\theta(X,Z)$. Take such a $Z$.
Then, there exists a coded $\beta_n$-model $\M$ such that $X,Y,Z,W \in \M$ and $\M \models \beta_n\RFN^{W}(\sigma(X) \lor \tau(X))$. 

We prove that $\M \models \beta_n\RFN^{W}(\tau(X))$. It is sufficient to show that 
$\forall a \in W\forall \K \in \M(Z \in \K \in H^{\sigma(X) \lor \tau(X)}_a \to \K \in H^{\tau(X)}_a )$. We will prove by induction on $a$.
Assume the desired condition holds for all $b < a$. Let $\K \in \M$ be such that $Z \in \K \in H^{\sigma(X)\lor \tau(X)}_a$. Since $X \in \K$ and $\beta_n(\K)$, it is enough to show that 
$\forall b < a \forall A \in \K \exists \K' \in \K (A \in \K' \in H^{\tau(X)}_b \land \K' \models \tau(X)) $ to prove $\K \in H^{\tau(X)}_a$.
Let $b < a$ and $A \in \K$. Then, there exists $\K' \in \K$ such that $A,Z \in \K' \in H^{\sigma (X) \lor \tau(X)}_b \land \K' \models (\sigma(X) \lor \tau(X))$. By the induction hypothesis, $\K' \in H^{\tau(X)}_b$.
Since $\K'$ is a $\beta_n$-model and $Z \in \K'$, $\K' \models \lnot \sigma(X)$. Thus, $\K' \models \tau(X)$.

(2) This is immediate from the definition.
\end{proof}

\begin{theorem}
  $\Sigma^1_{n+1}\hyp\TDC_0$ is $\Pi^1_{n+2}$-conservative over $\ATRo + \Pi^1_{n+2}\hyp\beta_n\RFN^{\ON}\hyp\Rule$.
\end{theorem}
\begin{proof}
  We can apply Theorem~\ref{BPconservation}. 
\end{proof}

\begin{definition}
  Let $\sigma(\V)$ be a formula such that $\ACAo \vdash \forall \V(\sigma(\V) \to \ATRo)$.
  For $r > 0$, define $(\beta_n\RFN^{\ON})^r(\sigma(\V))$ as follows.
  \begin{align*}
    &(\beta_n\RFN^{\ON})^1(\sigma(\V)) \equiv \beta_n\RFN^{\ON}(\sigma(\V)), \\ &(\beta_n\RFN^{\ON})^{r+1}(\sigma(\V)) \equiv \beta_n\RFN^{\ON}(\sigma(\V) + (\beta_n\RFN^{\ON})^r(\sigma(\V))).
  \end{align*}
\end{definition}

As usual, we obtain the following conservation result.
\begin{corollary}
  $\Th_{\Pi^1_{n+2}}(\Sigma^1_{n+1}\hyp\TDC_0) \equiv \ATRo + \{(\beta_n\RFN^{\ON})^r(\ATRo) : r > 0\}$.
\end{corollary}

By the same argument as in the proof of Theorem~\ref{thm Boolean conservatioin}, we obtain the following result.
\begin{theorem}
  $\Th_{\B(\Pi^1_{n+2})}(\Sigma^1_{n+1}\hyp\TDC_0) \equiv \ATRo + (\Pi^1_{n+2})^{-}\hyp\beta_n\RFN^{\ON}$.
\end{theorem}

By the same arguments as in the proofs of
Theorems~\ref{thm lower cons for pi},
\ref{thm lower cons for Sigma} and
\ref{thm lower cons for B},
we obtain the following results.
\begin{theorem}
  We have the following conservation results.
  \begin{itemize}
    \item For $m < n$,
    \begin{align*}
      \Th_{\Pi^1_{m+2}}
        (\Sigma^1_{n+1}\hyp\TDC_0)
      \equiv \ATRo +
      \bigl\{
        \beta_m\RFN
        (\ATRo + (\beta_n\RFN^{\ON})^r(\ATRo))
        : r > 0
      \bigr\}.
    \end{align*}

    \item For $m < n$,
    \begin{align*}
      &\Th_{\B(\Pi^1_{m+2})}
        (\Sigma^1_{n+1}\hyp\TDC_0)\\
      \equiv\quad &\ATRo +
      \bigl\{
        \sigma \to \beta_m\RFN
        (\ATRo + (\beta_n\RFN^{\ON})^r(\ATRo+\sigma))
        :\\
      &\hspace{8em}
        \sigma \in (\Pi^1_{m+2})^{-},\ r > 0
      \bigr\}.
    \end{align*}

    \item For $m \leq n$,
    \begin{align*}
      &\Th_{\Sigma^1_{m+2}}
        (\Sigma^1_{n+1}\hyp\TDC_0)\\
      \equiv\quad &\ATRo +
      \bigl\{
        \sigma \to \exists \M
        \bigl(
          \beta_m(\M) \land
          \M \models
          \ATRo + (\beta_n\RFN^{\ON})^r(\ATRo+\sigma)
        \bigr)
        :\\
      &\hspace{8em}
        \sigma \in (\Pi^1_{m+2})^{-},\ r > 0
      \bigr\}.
    \end{align*}
  \end{itemize}
\end{theorem}

\subsection{Transfinitely iterated reflection by Pakhomov and Walsh}

In \cite{PW}, Pakhomov and Walsh introduced another definition of transfinitely iterated reflection based on G\"odel's fixed-point theorem.
In what follows, we show the equivalence between our definition and the definition in \cite{PW}.

We begin by introducing the definition in \cite{PW}.
\begin{definition}
    Let $\{\varphi_e(\U,\V)\}_{e \in \omega}$ be an effective enumeration of $L_2$-formulas whose free variables are exactly $\U$ and $\V$.
    For $n \in \omega$, put 
    \begin{align*}
        \Psi_n(e,T,L) \equiv \forall \V \exists \M(\V,T,L \in \M \land \beta_n(\M) \land  \M \models T \land \forall b \in L (\M \models \varphi_e(T,L\mathord{\upharpoonright}_b))).
    \end{align*}
    Then, we define $\beta_n\RFN^{\mathsf{it}}(T,L)$ as a fixed point of $\Psi_n$. That is,
    \begin{align*}
        \ACAo \vdash \beta_n\RFN^{\mathsf{it}}(T,L) \leftrightarrow \Psi_n(\GN{\beta_n\RFN^{\mathsf{it}}},T,L).
    \end{align*}
\end{definition} 

\begin{remark}
    In general, $\sigma \in T$ or $\varphi_e$ may be nonstandard when we work in a non-$\omega$-model.
    Therefore, $\M \models T$ means that for each $\sigma \in T$, there exists an evaluation for $(\M,\sigma)$ that ensures $\M \models \sigma$.
    Throughout this subsection, we assume that our base theory includes $\ACA_0^+$. Therefore,
    we may assume that every coded $\omega$-model has a total evaluation.
\end{remark}

\begin{definition}
  For a formula $\sigma(\V)$, we define 
  $\beta_n\RFN^{\ON}_{\PW}$ as 
  \begin{align*}
    \beta_n\RFN^{\ON}_{\PW}(\sigma(\V)) \equiv \forall W(\WO(W) \to \beta_n\RFN^{\mathsf{it}}(\sigma(\V),W)).
  \end{align*}
    We also define $\Gamma\hyp\beta_n\RFN^{\ON}_{\PW} = \{\forall \V(\sigma(\V) \to \beta_n\RFN^{\ON}_{\PW}(\sigma(\V))): \sigma \in \Gamma\}$.
\end{definition}

\begin{remark}
We note that, in \cite{PW}, coded $\omega$-model reflection of $\sigma$ is defined as 
`if every coded $\omega$-model satisfies $\sigma$, then $\sigma$ is true'. 
However, we adopt its dual `if $\sigma$ is true, then it holds in some coded $\omega$-model' to align our definition 
with those of Simpson \cite{Simpson} and R\"{u}ede \cite{RuedeThesis,Ruede}.

We also note that, in our definition, every $\beta_n$-model is a model of $\ACAo$, even for $n = 0$.
Thus, $\beta_n\RFN^{\mathsf{it}}(\sigma(\V),W)$ is equivalent to $\beta_n\RFN^{\mathsf{it}}(\ACAo + \sigma(\V),W)$.
\end{remark}

\begin{lemma}
    Over $\ACAo$, $\beta_n\RFN^{\ON}_{\PW}(\ACAo)$ implies $\ATRo$.
\end{lemma}
\begin{proof}
  The case $n = 0$ is {\cite[Corollary 2.13]{PW}}.
  For $n > 0$, it is easy to see that $\beta_n\RFN^{\ON}_{\PW}(\ACAo)$ implies $\beta_0\RFN^{\ON}_{\PW}(\ACAo)$.
\end{proof}

\begin{theorem}[$\ACAo$]
    The following are equivalent.
    \begin{enumerate}
        \item $\ATRo + \Pi^1_{n+2}\hyp\beta_n\RFN^{\ON}$.
        \item $\Pi^1_{n+2}\hyp\beta_n\RFN^{\ON}_{\PW}$.
    \end{enumerate}
\end{theorem}
\begin{proof}
    $(1 \to 2)$
    Assume $\ATRo + \Pi^1_{n+2}\hyp\beta_n\RFN^{\ON}$.
    Let $\sigma(\V)$ be a $\Pi^1_{n+2}$-formula.
    Let $A$ be a set such that $\sigma(A)$ holds.
    Let $W$ be a well-order that has at least two elements.

    We will show that $\beta_n\RFN^{\mathsf{it}}(\sigma(A),W)$ holds.
    Let $X$ be a set.
    By $\Pi^1_{n+2}\hyp\beta_n\RFN^{\ON}$, take a coded $\beta_n$-model $\M$ such that $A,X,W \in \M \land \M \models \beta_n\RFN^{W}(\sigma(A))$. That is,
    $$\forall a \in W \forall Y \in \M \exists \K \in \M (A,X,Y,W \in \K \land \beta_n(\K) \land  \K  \in H^{\sigma(A)}_{a} \land \K \models  \sigma(A)).$$
    We will prove that the following holds for any $a \in W$: If $\K \in \M$ and $A,W \in \K \land \beta_n(\K)$, then 
    \begin{align*}
        \K \in H^{\sigma(A)}_{a} \leftrightarrow \forall b <_W a (\K \models \beta_n\RFN^{\mathsf{it}}(\sigma(A), W\mathord{\upharpoonright}_b)).
    \end{align*}
    Assume that the desired condition holds for any $b<_W a$. 
    Let $\K \in \M$ be such that $A,W \in \K$.

    $(\to)$ Assume $\K \in H^{\sigma(A)}_{a}$.
    Then, by definition of $H^{\sigma(A)}_a$, we obtain 
    $$\forall b <_W a \forall Y \in \K \exists \K' \in \K(A,W,Y \in \K' \land \K' \in H^{\sigma(A)}_{b} \land \K' \models \sigma(A) ).$$
    Let $b <_W a$. We will prove that $\K \models \beta_n\RFN^{\mathsf{it}}(\sigma(A),W\mathord{\upharpoonright}_b)$, which is equivalent to 
    \begin{align*}
        \K \models \forall Y \exists \K' (A,W,Y \in \K' \land \beta_n(\K') \land \K' \models \sigma(A)  \land \forall c <_W b(\K' \models \beta_n\RFN^{\mathsf{it}}(\sigma(A),(W\mathord{\upharpoonright}_b)\mathord{\upharpoonright}_{c}) )).
    \end{align*}
    Let $Y \in \K$. Then, there exists $\K' \in \K$ such that $A,W,Y \in \K' \in H^{\sigma(A)}_{b} \land \K' \models \sigma(A)$.
    By the induction hypothesis, $\forall c <_W b (\K' \models \beta_n\RFN^{\mathsf{it}}(\sigma(A),W\mathord{\upharpoonright}_{c} ))$.
    This yields the desired condition because $W\mathord{\upharpoonright}_{c} = (W\mathord{\upharpoonright}_b)\mathord{\upharpoonright}_{c}$.

    $(\leftarrow)$ Assume $\forall b <_W a (\K \models \beta_n\RFN^{\mathsf{it}}(\sigma(A), W\mathord{\upharpoonright}_b))$.
    Let $b <_W a$. We will prove that $$\forall Y \in \K \exists \K' \in \K(A,Y,W \in \K' \land \beta_n(\K') \land \K' \in H^{\sigma(A)}_{b} \land \K' \models \sigma(A)).$$
    Since $\K \models \beta_n\RFN^{\mathsf{it}}(\sigma(A),W\mathord{\upharpoonright}_b)$, we obtain 
    $$
    \forall Y \in \K \exists \K' \in \K (A,W,Y \in \K' \land \beta_n(\K') \land \K'\models \sigma(A) \land \forall c {<_W} b\, \K' \models  \beta_n\RFN^{\mathsf{it}}(\sigma(A),(W\mathord{\upharpoonright}_b)\mathord{\upharpoonright}_{c})),
    $$
    which, by the induction hypothesis and the equality $W\mathord{\upharpoonright}_{c} = (W\mathord{\upharpoonright}_b )\mathord{\upharpoonright}_{c}$, is equivalent to
    $$
    \forall Y \in \K \exists \K' \in \K (A,W,Y \in \K' \land \K \models \beta_n(\K') \land \K' \models \sigma(A) \land \K' \in H^{\sigma(A)}_b).
    $$

    ($2 \to 1$) Essentially the same argument as for $(1 \to 2)$ works.
\end{proof}

\section{Model-theoretic proof}
In the previous sections, we have proved the following conservation results by a proof-theoretic method.
\begin{alignat*}{2}
  &\Th_{\Pi^1_{n+2}}(\Sigma^1_{{n+1}}\hyp\DC_0) &&\equiv \ACAo + \beta_n\RFN^{\omega}(\ACAo), \\
  &\Th_{\Pi^1_{n+2}}(\Sigma^1_{{n+1}}\hyp\TDC_0) &&\equiv \ACAo + \beta_n\TRFN^{\omega}(\ACAo), \\
  &\Th_{\B(\Pi^1_{n+2})}(\Sigma^1_{n+1}\hyp\DC_0) &&\equiv (\Pi^1_{n+2})^{-}\hyp\beta_n\RFN_0, \\
  &\Th_{\B(\Pi^1_{n+2})}(\Sigma^1_{n+1}\hyp\TDC_0) &&\equiv (\Pi^1_{n+2})^{-}\hyp\beta_n\TRFN_0.
\end{alignat*}
In this section, we give a model-theoretic proof of these.

\begin{definition}
  In this definition, we consider models that are not necessarily coded.
  Let $\M,\M'$ be models of $\ACAo$. We say $\M$ is a $\beta_n$-submodel of $\M'$ if $\M$ is a submodel of $\M'$ having the same first-order part as $\M'$, and $\M$ and $\M'$ agree on the truth of $\Pi^1_{n}$-formulas with parameters from $\M$. We write $\M \subseteq_{\beta_n} \M'$ to mean that $\M$ is a $\beta_n$-submodel of $\M'$.
\end{definition}

The following lemma is standard in the model theory of arithmetic.
\begin{lemma}\label{lem conservation general}
  Let $T$ and $T'$ be $L_2$-theories including $\ACAo$. Assume that for any countable model $\M' \models T'$ and second-order parameters $\vec{X} \in \M'$, there exists $\M \subseteq_{\beta_n} \M'$ such that $\vec{X} \in \M \land \M \models T$. 
  Then, any $\Pi^1_{n+2}$-sentence provable in $T$ is provable in $T'$. 

  In addition, we can restrict $\M'$ to recursively saturated models.
\end{lemma}
\begin{proof}
  Let $\forall \V \exists \U \theta(\V,\U)$ be a $\Pi^1_{n+2}$-sentence such that $T' \not \vdash \forall \V \exists \U \theta(\V,\U)$. We will show that $T \not \vdash \forall \V \exists \U \theta(\V,\U)$.

  Since $T' + \exists \V \forall \U \lnot \theta(\V,\U)$ is consistent, there is a countable $\M' \models T' + \exists \V \forall  \U \lnot \theta(\V,\U)$.
  Take an $X \in \M'$ such that $\M' \models \forall \U \lnot \theta(X,\U)$. By assumption, 
  take $\M \subseteq_{\beta_n} \M'$ such that $X \in \M$ and $\M \models T$. 
  Then, $\M \models \forall \U \lnot \theta(X,\U)$ and hence $\M \models \lnot \forall \V \exists \U \theta(\V,\U)$. Thus, $T + \lnot \forall \V \exists \U \theta(\V,\U)$ has a model.

  For the additional part, we note that $\M'$ has a countable recursively saturated elementary extension $\M''$. Using $\M''$ instead of $\M'$, the same proof works.
\end{proof}

\begin{lemma}\label{Lem Shrink models}
  Let $\M$ be a recursively saturated model of $\ACAo + \beta_n\RFN^{\omega}(\ACAo)$.
  Then, for any $X \in \M$
  there exists $\langle \NN_i \rangle_{i \in \N^{\M}} \in \M$ with the following property:
  $\M \models \forall i \in \N(X \in \NN_i \in \NN_{i+1} \land \beta_n(\NN_i) \land \NN_i \models \beta_n\RFN^{r}(\ACAo))$ for all $r \in \omega, r > 0$.
\end{lemma}
\begin{proof}
  Let $X \in \M$.
  We will show that $\{ \forall i \in \N(X \in \NN_i \in \NN_{i+1} \land \beta_n(\NN_i) \land \NN_i \models \beta_n\RFN^{r}(\ACAo)) : r > 0\}$ is a recursive type over $(\M,X)$.
  It is enough to show that $\M \models \exists \langle \NN_i \rangle_{i}  \forall i \in \N(X \in \NN_i \in \NN_{i+1} \land \beta_n(\NN_i) \land \NN_i \models \beta_n\RFN^{r}(\ACAo)) $ for all $r > 0$.

  Fix an $r > 0$. We work in $\M$. There exists $\K$ such that $\beta_n(\K) \land X \in \K \land \K \models \beta_n\RFN^{r+1}(\ACAo)$. 
  By $\beta_n\RFN^{r+1}(\ACAo)$ in $\K$, there exists $\NN_0 \in \K$ such that $\K \models ( X \in \NN_0 \land \beta_n(\NN_0) \land \NN_0 \models \beta_n\RFN^r(\ACAo))$. Since $\K$ is a $\beta_n$-model, $X \in \NN_0 \land \beta_n(\NN_0) \land \NN_0 \models \beta_n\RFN^r(\ACAo)$ actually holds.
  Again, by $\beta_n\RFN^{r+1}(\ACAo)$ in $\K$, we obtain $\NN_1 \in \K$ such that 
  $\NN_0 \in \NN_1 \land \beta_n(\NN_1) \land \NN_1 \models \beta_n\RFN^{r}(\ACAo)$.
  Repeating this argument, we obtain the desired sequence $\langle \NN_i \rangle_{i \in \N}$.
  We note that $\langle \NN_i \rangle_{i \in \N}$ is arithmetically definable from $\K$.
\end{proof}

\begin{lemma}
  Let $\M$ be a recursively saturated model of $\ACAo$ and $\K \in \M$ be a coded $\beta_n$-model.
  Then $\K$ is recursively saturated.
\end{lemma}
\begin{proof}
  Let $X \in \K$. Let $p(\V)$ be a recursive type over $(\K,X)$. 
  Let $\varphi_0(\V,X),\ldots,\varphi_n(\V,X) \in p(\V)$.
  Then $\K \models \exists \V(\varphi_0(\V,X) \land \cdots \land \varphi_n(\V,X))$.
  Thus, $\M \models \exists \V \in \K [\K \models \varphi_0(\V,X) \land \cdots \land \varphi_n(\V,X)]$.
  Hence, $q(\V) = \{\V \in \K \land \K \models \varphi(\V,X) : \varphi \in p(\V)\}$ is a recursive type over $(\M,X,\K)$. Since $\M$ is recursively saturated, $q(\V)$ is realizable by some $Y \in \M$.
  For such a $Y$, $\M \models [Y \in \K \land \K \models \varphi(Y,X)]$ for all $\varphi \in p(\V)$. 
  Hence $\K \models \varphi(Y,X)$ for all $\varphi \in p(\V)$. 
  Thus, $p(\V)$ is realized in $\K$.
\end{proof}

\begin{theorem}\label{thm compact consv}
  Any $\Pi^1_{n+2}$-formula provable in $\Sigma^1_{n+1}\hyp\DC_0$ is provable in $\ACAo + \beta_n\RFN^{\omega}(\ACAo)$.
\end{theorem}
\begin{proof}
  Let $\M'$ be a countable recursively saturated model of $\ACAo + \beta_n\RFN^{\omega}(\ACAo)$.
  We will show that for any $P \in \M'$, there exists $\M \subseteq_{\beta_n} \M'$ such that 
  $P \in \M$ and $\M \models \Sigma^1_{n+1}\hyp\DC_0$.

  Let $P \in \M'$. We construct $\beta_n$-submodels $\L_i$ and $\K_i$ of $\M'$ as follows.
  In what follows, let the second-order part of $\L_i$ be enumerated as $\{(\L_i)_m : m \in \omega\}$, and let $\{\theta_e : e \in \omega\}$ be an enumeration of $\Pi^1_{n}$-formulas with exactly three free set variables.
  \begin{itemize}
    \item $\L_i \subseteq \L_{i+1},\;\K_{i+1} \subseteq \K_{i},\; \L_i \in \K_i$ for all $i \in \omega$.
    \item $P \in \L_0$.
    \item $\K_i \models \beta_n\RFN^{\omega}(\ACAo)$ and $\K_i$ is recursively saturated for all $i \in \omega$.
    \item For $s = \langle e,i,m,m' \rangle$, either $\L_{s+1} \models \exists H(H_0 = (\L_i)_m \land \forall x \theta_e(H_x,H_{x+1},(\L_{i})_{m'})  )$ or $\K_{s+1} \models \forall \U \lnot \theta_e(X,\U,(\L_i)_{m'})$ for some $X \in \L_{s+1}$.
  \end{itemize}

  First, put $\K_0 = \M'$. Since $P \in \M'$ and $\M' \models \beta_n\RFN(\ACAo)$, there exists $\L_0 \in \M'$ such that $P \in \L_0$ and $\L_0 \models \ACAo$ and $\M' \models (\beta_n(\L_0))$.

  Assume that $s = \langle e,i,m,m' \rangle$ and  $\L_{0},\ldots,\L_{s}, \K_0,\ldots,\K_{s}$ are already defined. Let $\langle \NN_x \rangle_{x} \in \K_s$ be such that $\K_s \models \forall x [\L_s \in \NN_x \in \NN_{x+1} \land \beta_n(\NN_{x}) \land \NN_{x} \models \beta_n\RFN^{r}(\ACAo)]$ for all $r >0$.
  Within $\K_s$, each $\NN_x$ is regarded as the sequence $\langle (\NN_x)_y \rangle_{y \in \N}$.
  We try to make a sequence $\langle H_x \rangle_{x \in \N} \in \K_s$ such that 
  \begin{itemize}
    \item $H_0 = (\L_i)_m$,
    \item $\forall x \in \N$,  $H_{x+1} = (\NN_{x})_y$ where $y = \mu z. \NN_{x} \models \theta_e(H_x,(\NN_{x})_z,(\L_i)_{m'})$.
  \end{itemize}
  
  [Case 1: $H$ exists.]
  In this case, $H$ satisfies $\K_s \models \forall x \theta_e(H_x,H_{x+1},(\L_i)_{m'})$ because each $\NN_{x}$ is a $\beta_n$-model. 
  We define $\L_{s+1}$ and $\K_{s+1}$ as 
  \begin{itemize}
    \item $\L_s, H \in \L_{s+1} \in \K_s$ such that $\K_s \models \beta_n(\L_{s+1})$.
    \item $\K_{s+1} = \K_s$.
  \end{itemize}
  Since $\K_s \models \beta_n\RFN^{\omega}(\ACAo)$, such an $\L_{s+1}$ exists. 

  \vspace*{0.5em}

  [Case 2: The construction of $H$ fails.]
  Let $x \in \N$ be the smallest one such that 
  $$\NN_x \models \lnot \exists \U \theta_e(H_x,\U,(\L_i)_{m'}).$$
  Put $\K_{s+1} = \NN_x$. Then, take $\L_{s+1} \in \K_{s+1}$ such that 
  $\L_{s},H_x \in \L_{s+1}$ and $\K_{s+1} \models \beta_n(\L_{s+1})$. Since $\K_{s+1} \models \beta_n\RFN^{\omega}(\ACAo)$, such an $\L_{s+1}$ exists.
 
  In both cases, it is easy to see that $\L_{s+1}$ and $\K_{s+1}$ satisfy the desired conditions.

  We define $\M = \bigcup_{i \in \omega} \L_i$.
  Then, $P \in \M$. Since $\L_i \subseteq_{\beta_n} \M'$ and $\L_i \subseteq \L_{i+1}$ for all $i$, $\M \subseteq_{\beta_n} \M'$.
  To prove $\M \models \Sigma^1_{n+1}\hyp\DC_0$, let $\theta_e$ be a $\Pi^1_{n}$-formula and $B \in \M$ be such that $\M \models \forall \V \exists \U \theta_e(\V,\U,B)$.
  We show that $\M \models \forall A \exists H (H_0 = A \land \forall x \theta_e(H_x,H_{x+1},B))$.
  
  Let $A \in \M$.
  Then, there exist $i,m,m' \in \omega$ such that 
  $A = (\L_i)_m,B = (\L_i)_{m'}$. Put $s = \langle e,i,m,m' \rangle$. 
  By the construction of $\L_{s+1}$, either $\L_{s+1} \models \exists H(H_0 = (\L_i)_m \land \forall x \theta_e(H_x,H_{x+1},(\L_{i})_{m'})  )$ or $\K_{s+1} \models \forall \U \lnot \theta_e(X,\U,(\L_i)_{m'})$ for some $X \in \L_{s+1}$.
  We show that the latter never happens.
  Since $\K_{s+1} \subseteq_{\beta_n} \M'$ and $\M \subseteq_{\beta_n} \M'$
  and $\K_{s+1} \supseteq \L_{j}$ for all $j \in \omega$, we obtain $\M \subseteq_{\beta_n} \K_{s+1}$.
  Thus, if $\K_{s+1} \models \forall \U \lnot \theta_e(X,\U,(\L_i)_{m'})$ for an $X \in \L_{s+1}$, then 
  $\M \models \forall \U \lnot \theta_e(X,\U,(\L_i)_{m'})$. However, $\M \models \forall \V \exists \U\theta_e (\V,\U,(\L_i)_{m'})$.

  Now $\L_{s+1} \models \exists H(H_0 = A \land \forall x \theta_e(H_x,H_{x+1},B))$. Since $\L_{s+1} \subseteq_{\beta_n}\M$, $\M \models \exists H(H_0 = A \land \forall x \theta_e(H_x,H_{x+1},B))$.
  This completes the proof.
\end{proof}

We next show the conservation for $\Sigma^1_{n+1}\hyp\TDC_0$.
The proof is essentially the same as for $\Sigma^1_{n+1}\hyp\DC_0$, but we need a slight modification.
First, we prepare a new lemma corresponding to Lemma~\ref{Lem Shrink models}.

\begin{lemma}\label{lem shrink trfn}
  Let $\M$ be a recursively saturated model of $\ACAo + \beta_n\TRFN^{\omega}(\ACAo)$ and $W \in \M$ be such that $\M \models \WO(W)$.
  Then, for any $X \in \M$, there exists $\langle \NN_a \rangle_{a \in W} \in \M$ with the following property:
  $\M \models \forall a \in W(X, \NN_{<a} \in \NN_{a} \land \beta_n(\NN_a) \land \NN_a \models \beta_n\TRFN^{r}(\ACAo))$ for all $r \in \omega, r > 0$.
\end{lemma}
\begin{proof}
  Let $\M$ be a recursively saturated model of $\ACAo + \beta_n\TRFN^{\omega}(\ACAo)$ and $X,W \in \M$ be such that $\M \models \WO(W)$.
  We will show that $\{\forall a \in W(X,\NN_{<a} \in \NN_{a} \land \beta_n(\NN_a) \land \NN_a \models \beta_n\TRFN^{r}(\ACAo)) : r >0\}$ is a recursive type over $(\M,X,W)$.

  Fix an $r > 0$. Since $\M \models \beta_n\TRFN^{r+1}(\ACAo)$, 
  $\forall a \in W(X, \NN_{<a} \in \NN_{a} \land \beta_n(\NN_a) \land \NN_a \models \beta_n\TRFN^{r}(\ACAo))$ is realizable in $\M$. This completes the proof.
\end{proof}

\begin{theorem}
  Any $\Pi^1_{n+2}$-formula provable in $\Sigma^1_{n+1}\hyp\TDC_0$ is provable in $\ACAo + \beta_n\TRFN^{\omega}(\ACAo)$.
\end{theorem}
\begin{proof}
  Let $\M'$ be a countable recursively saturated model of $\ACAo + \beta_n\TRFN^{\omega}(\ACAo)$. We will show that for any $P \in \M'$, there exists $\M \subseteq_{\beta_n} \M'$ such that $P \in \M$ and $\M \models \Sigma^1_{n+1}\hyp\TDC_0$.

  In what follows, let the second-order part of $\L_i$ be enumerated as $\{(\L_i)_m : m \in \omega\}$, and let $\{\theta_e : e \in \omega\}$ be an enumeration of $\Pi^1_{n}$-formulas with exactly three free set variables.
  \begin{itemize}
    \item $\L_i \subseteq \L_{i+1}, \K_{i+1} \subseteq \K_i, \L_i \in \K_i$ for all $i \in \omega$.
    \item $P\in \L_0$.
    \item $\K_i \models \beta_n\TRFN^{\omega}(\ACAo)$ and $\K_i$ is recursively saturated for all $i \in \omega$.
    \item For $s = \langle e,i,m,m',m'' \rangle$ one of the following holds.
    \begin{enumerate}
      \item $\L_{s+1} \models \lnot \WO((\L_i)_{m''})$.
      \item $\L_{s+1}$ believes that $(\L_i)_{m''}$ is a well-order with a maximum element.
      \item $\L_{s+1} \models \WO((\L_i)_{m''}) \land \exists H(H_0 = (\L_i)_m \land \forall a \in ((\L_i)_{m''})^{+} (\theta_e(H_{<a},H_a,(\L_i)_{m'})) )$.
      \item $\K_{s+1} \models \forall \U \lnot \theta_e(X,\U,(\L_i)_{m'}) $ for some $X \in \L_{s+1}$.
    \end{enumerate}
  \end{itemize}
  First, put $\K_0 = \M'$. Then, take $\L_0 \in \M'$ such that 
  $\M' \models P \in \L_0 \land \L_0 \models \ACAo \land \beta_n(\L_0)$.

  Assume that $s = \langle e,i,m,m',m'' \rangle$ and $\L_0,\ldots,\L_s,\K_0,\ldots,\K_s$ are already defined.
  
  If $\K_s$ believes that $(\L_i)_{m''}$ is not a linear order, then we put $\K_{s+1} = \K_s$ and $\L_{s+1} = \L_s$.

  If $\K_s$ believes that $(\L_i)_{m''}$  is an ill-founded linear order, then
take a descending sequence $f$ of $(\L_i)_{m''}$. Put $\K_{s+1} = \K_s$, and choose
  $\L_{s+1} \in \K_{s+1}$ such that $\L_{s},f \in \L_{s+1}$ and $\K_{s+1} \models \beta_n(\L_{s+1})$. 

  If $\K_s$ believes $(\L_i)_{m''}$ is a well-order with a maximum element, then we put $\L_{s+1} = \L_s$ and $\K_{s+1} = \K_s$.
  
  Assume $\K_s$ believes $(\L_i)_{m''}$ is a well-order having no maximum. Let $W= (\L_i)_{m''}$.
  By Lemma~\ref{lem shrink trfn}, take $\langle \NN_a \rangle_{a \in W}$ such that 
  \begin{align*}
    \K_s \models \forall a \in W(\L_{s},\NN_{<a} \in \NN_a \land \beta_n(\NN_a) \land \NN_a \models \beta_n\TRFN^{r}(\ACAo) ) \text{ for all } r \in \omega, r >0.
  \end{align*}
  We note that $\NN_a \models \Pi^1_{n}\hyp\TR_0$ for all $a \in W$. This follows from $\NN_a \models \beta_n\TRFN^{\omega}(\ACAo)$ by Theorem~\ref{thm TRFN vs TR}.
  We define a double-indexed sequence $\langle \langle H^a_b \rangle_{b < a} : a \in W \rangle \in \K_s$.
  For each $a \in W^+$, we define $\langle H^a_{b} \rangle_{b < a} \in \NN_a$ as follows.
  \begin{itemize}
    \item $H^a_0 = (\L_i)_m$.
    \item for $0 < b < a$, $$H^a_{b} = \begin{cases}
      (\NN_b)_y \text{ where } y = \min \{z  : \NN_a \models \theta_e(H^a_{<b},(\NN_b)_z ,(\L_i)_{m'})\} & \text{ if $y$ exists} \\
      \varnothing & \text{otherwise}
      \end{cases}.$$
  \end{itemize}
  Since $H^a_b$ is $\A(\Pi^1_{n})$-definable\footnote{Here, $\A(\Gamma)$ denotes the set of formulas generated by $\forall x \in \N, \land, \lnot$ from $\Gamma$.} from $H^a_{<b}$ and $\NN_b \in \NN_a$, this definition works.
  Then, $\K_s \models \forall a \in W(\forall b,a' < a (b < a' \to H^a_b = H^{a'}_b ))$ by induction on $a$.
  Define $\langle H_a \rangle_{a \in W} \in \K_{s}$ by $H_a = H^{a+1}_a$.
  Then, there are two possible cases.

  [Case 1 $\K_s \models \forall a \in W^+(\theta_e(H_{<a},H_a,(\L_i)_{m'}))$.]
  In this case, put $\K_{s+1} = \K_s$, and choose $\L_{s+1} \in \K_{s}$ such that 
  $\L_s, \langle H_a \rangle_{a \in W} \in \L_{s+1} \land \K_s \models \beta_n(\L_{s+1})$.

  [Case 2 $\K_s \models \exists a \in W^+ \lnot \theta_e(H_{<a},H_a,(\L_i)_{m'})$.]
  Since $\{a \in W^+ : \lnot \theta_e(H_{<a},H_a, (\L_i)_{m'})\}$ is $\Sigma^1_{n}$-definable, it exists in $\K_{s}$. Then, there exists a minimum element $a \in W^+$ such that 
  $\K_s \models \lnot \theta_e(H_{<a},H_a,(\L_i)_{m'})$.
  Put $\K_{s+1} = \NN_a$. Then, by the construction of $H_a$, $\NN_a \models \lnot \exists \U \theta_{e}(H_{<a},\U,(\L_i)_{m'})$. Let $\L_{s+1} \in \NN_{a}$ be such that 
  $\L_{s},H_{<a} \in \L_{s+1}$ and $\NN_a \models \beta_n(\L_{s+1})$.
  This completes the construction of $\L_{s+1}$ and $\K_{s+1}$.

  Put $\M = \bigcup_{i \in \omega} \L_i$. Then, for the same reason as in Theorem~\ref{thm compact consv}, $\M \models \Sigma^1_{n+1}\hyp\TDC_0$. This completes the proof.
\end{proof}

We next see that the same proof works for $\B(\Pi^1_{n+2})$-conservation.

\begin{lemma}\label{lem consv for boole}
  Let $T$ and $T'$ be $L_2$-theories including $\ACAo$. Assume that for any countable recursively saturated model $\M' \models T'$, second-order parameters $\vec{X} \in \M'$ and $\Pi^1_{n+2}$-sentence $\eta$ such that $\M' \models \eta$, there exists $\M \subseteq_{\beta_n} \M'$ such that $\vec{X} \in \M \land \M \models T + \eta$. 
  Then, any $\B(\Pi^1_{n+2})$-sentence provable in $T$ is provable in $T'$. 
\end{lemma}
\begin{proof}
  To show $\B(\Pi^1_{n+2})$-conservation, it is enough to show that 
  for any $\sigma \in \Sigma^1_{n+2}$ and $\pi \in \Pi^1_{n+2}$, if 
  $T \vdash \sigma \lor \pi$, then $T' \vdash \sigma \lor \pi$.

  Assume $T' \not \vdash \sigma \lor \pi$.
  Then $T' + \lnot \sigma + \lnot \pi$ has a recursively saturated model $\M'$.
  By the same argument as in Lemma~\ref{lem conservation general}, we obtain a model $\M$ of $T + \lnot \sigma + \lnot \pi$. Thus, $T \not \vdash \sigma \lor \pi$. This completes the proof. 
\end{proof}

We mimic the proof of Theorem~\ref{thm compact consv} for $\Sigma^1_{n+1}\hyp\DC_0$ and $\ACAo + (\Pi^1_{n+2})^{-}\hyp\beta_n\RFN$.
The following lemma corresponds to Lemma~\ref{Lem Shrink models}.

\begin{lemma}
  Let $\eta$ be a $\Pi^1_{n+2}$-sentence.
  Let $\M$ be a recursively saturated model of $\ACAo + \beta_n\RFN^{\omega}(\eta)$. Then, for any $X \in \M$, there exists $\langle \NN_x \rangle_{x \in \N^{\M}} \in \M$ with the following property:
  $$
  \M \models \forall x (X \in \NN_x \in \NN_{x+1} \land \beta_n(\NN_x) \land \NN_x \models \eta + \beta_n\RFN^r(\eta)) \text{ for all } r > 0.
  $$
\end{lemma}
\begin{proof}
  Essentially the same as for Lemma~\ref{Lem Shrink models}.
\end{proof}
To apply this lemma, we use the following property of $(\Pi^1_{n+2})^{-}\hyp\beta_n\RFN$.
\begin{lemma}
  Let $\eta$ be a $\Pi^1_{n+2}$-sentence. Then, $\ACAo + (\Pi^1_{n+2})^{-}\hyp\beta_n\RFN$ proves $\eta \to \beta_n\RFN^r(\eta)$ for all $r > 0$.
\end{lemma}
\begin{proof}
  We note that $\beta_n\RFN^r(\eta)$ is again a $\Pi^1_{n+2}$-sentence. 
  Thus, we obtain $\ACAo + (\Pi^1_{n+2})^{-}\hyp\beta_n\RFN \vdash \eta \to \beta_n\RFN^r(\eta)$ by induction on $r$.
\end{proof}

\begin{theorem}
  Any $\B(\Pi^1_{n+2})$-sentence provable in $\Sigma^1_{n+1}\hyp\DC_0$ is already provable in $\ACAo + (\Pi^1_{n+2})^{-}\hyp\beta_n\RFN$.
\end{theorem}
\begin{proof}
  The proof is essentially the same as for Theorem~\ref{thm compact consv}.
  To use Lemma~\ref{lem consv for boole}, we first pick a $\Pi^1_{n+2}$-sentence $\eta$ that is true in $\M'$.
  Then, we add a new condition for the definition of $\L_{i}$ and $\K_{i}$ so that $\L_{i} \models \eta$ and $\K_{i} \models \eta + \beta_n\RFN^{r}(\eta)$ for all $r > 0$. This construction works because of the above lemmas.
Since $\eta$ is $\Pi^1_{n+2}$ and $\L_i$ are $\beta_n$-submodels of $\M'$ such that $\L_0 \subseteq \L_1 \subseteq \L_2 \subseteq \cdots$, the constructed model $\M = \bigcup_{i \in \omega} \L_i \models \eta$.
\end{proof}

\begin{remark}
We also note that the conservation result
$$\Th_{\B(\Pi^1_{n+2})}(\Sigma^1_{n+1}\hyp\TDC_0) \equiv  (\Pi^1_{n+2})^{-}\hyp\beta_n\TRFN_0$$ is provable in the same way as the previous theorem.
\end{remark}

\section{Summary and speed-up}
We summarize the main conservation results proved in this paper.
Each item below specifies a theory $T$, a formula class $\Gamma$,
a base theory $B$, and a reflection axiom schema $R$ such that
\begin{equation*}
  \Th_{\Gamma}(T) \equiv B+R.
\end{equation*}
Items~(\ref{sum:dc:first})--(\ref{sum:dc:last}) were proved in
Section~2 for $T = \Sigma^1_{n+1}\hyp\DC_0$, and
items~(\ref{sum:tdc:first})--(\ref{sum:tdc:last}) were proved in
Section~3 for $T = \Sigma^1_{n+1}\hyp\TDC_0$.

\begin{enumerate}[label=(\arabic*),ref=\arabic*]
  \item\label{sum:dc:first}
    $\Gamma=\Pi^1_{n+2}$, and $B=\ACAo$, with
    \[
      R = \beta_n\RFN^{\omega}(\ACAo).
    \]

  \item For $m<n$, 
    $\Gamma=\Pi^1_{m+2}$, and $B=\ACAo$, with
    \[
      R = \bigl\{\beta_m\RFN(\beta_n\RFN^r(\ACAo)):r>0\bigr\}.
    \]

  \item 
    $\Gamma=\B(\Pi^1_{n+2})$, and $B=\ACAo$, with
    \[
      R = (\Pi^1_{n+2})^{-}\hyp\beta_n\RFN.
    \]

  \item For $m<n$, 
    $\Gamma=\B(\Pi^1_{m+2})$, and $B=\ACAo$, with
    \[
      R = \bigl\{
        \sigma\to\beta_m\RFN(\beta_n\RFN^r(\sigma)):
        \sigma\in(\Pi^1_{m+2})^{-},\ r>0
      \bigr\}.
    \]

  \item For $m\leq n$,
    $\Gamma=\Sigma^1_{m+2}$, and $B=\ACAo$, with
    \[
      R = \bigl\{
        \sigma\to\exists\M
        (\beta_m(\M)\land\M\models\beta_n\RFN^r(\sigma)):
        \sigma\in(\Pi^1_{m+2})^{-},\ r>0
      \bigr\}.
    \]

  \item\label{sum:dc:last} For $m\leq n$,
    $\Gamma=\Sigma^1_{m+1}$, and $B=\ACAo$, with
    \[
      R = \bigl\{
        \exists\M(\beta_m(\M)\land
        \M\models\beta_n\RFN^r(\ACAo)):r>0
      \bigr\}.
    \]

  \item\label{sum:tdc:first}
    $\Gamma=\Pi^1_{n+2}$, and $B=\ACAo$, with
    \[
      R = \beta_n\TRFN^{\omega}(\ACAo).
    \]

  \item For $m<n$, 
    $\Gamma=\Pi^1_{m+2}$, and $B=\ACAo$, with
    \[
      R = \bigl\{\beta_m\RFN(\beta_n\TRFN^r(\ACAo)):r>0\bigr\}.
    \]

  \item 
    $\Gamma=\B(\Pi^1_{n+2})$, and $B=\ACAo$, with
    \[
      R = (\Pi^1_{n+2})^{-}\hyp\beta_n\TRFN.
    \]

  \item For $m<n$, 
    $\Gamma=\B(\Pi^1_{m+2})$, and $B=\ACAo$, with
    \[
      R = \bigl\{
        \sigma\to\beta_m\RFN(\beta_n\TRFN^r(\sigma)):
        \sigma\in(\Pi^1_{m+2})^{-},\ r>0
      \bigr\}.
    \]

  \item For $m\leq n$,
    $\Gamma=\Sigma^1_{m+2}$, and $B=\ACAo$, with
    \[
      R = \bigl\{
        \sigma\to\exists\M
        (\beta_m(\M)\land\M\models\beta_n\TRFN^r(\sigma)):
        \sigma\in(\Pi^1_{m+2})^{-},\ r>0
      \bigr\}.
    \]

  \item 
    $\Gamma=\Pi^1_{n+2}$, and $B=\ATRo$, with
    \[
      R = \bigl\{(\beta_n\RFN^{\ON})^r(\ATRo):r>0\bigr\}.
    \]

  \item For $m<n$, 
    $\Gamma=\Pi^1_{m+2}$, and $B=\ATRo$, with
    \[
      R = \bigl\{
        \beta_m\RFN(\ATRo+
        (\beta_n\RFN^{\ON})^r(\ATRo)):r>0
      \bigr\}.
    \]

  \item 
    $\Gamma=\B(\Pi^1_{n+2})$, and $B=\ATRo$, with
    \[
      R = (\Pi^1_{n+2})^{-}\hyp\beta_n\RFN^{\ON}.
    \]

  \item For $m<n$, 
    $\Gamma=\B(\Pi^1_{m+2})$, and $B=\ATRo$, with
    \[
      R = \begin{aligned}[t]
        \bigl\{
        \sigma\to\beta_m\RFN
        (\ATRo+(\beta_n\RFN^{\ON})^r(\ATRo+\sigma)):\qquad\\
        \sigma\in(\Pi^1_{m+2})^{-},\ r>0
        \bigr\}.
      \end{aligned}
    \]

  \item\label{sum:tdc:last} For $m\leq n$,
    $\Gamma=\Sigma^1_{m+2}$, and $B=\ATRo$, with
    \[
      R = \begin{aligned}[t]
        \bigl\{
        \sigma\to\exists\M
        (\beta_m(\M)\land
        \M\models\ATRo+
        (\beta_n\RFN^{\ON})^r(\ATRo+\sigma)):\qquad\\
        \sigma\in(\Pi^1_{m+2})^{-},\ r>0
        \bigr\}.
      \end{aligned}
    \]

\end{enumerate}

In each case, $T$ has iterated exponential speed-up over $B+R$.
In what follows, we give the details.
We explain the speed-up in cases~(1), (3) and
(5); the remaining cases are treated similarly.

We first consider case~(1).
We note that $\beta_n\RFN^r(\ACAo)$ is defined by primitive recursion.
Thus, there is a $\Pi^1_{n+2}$-formula $\varphi(x)$ such that 
$\ACAo \vdash \varphi(\overline{r}) \leftrightarrow \beta_n\RFN^r(\ACAo)$ for all $r \in \omega, r > 0$. Here, $\overline{r}$ denotes the numeral for $r$.
Then, it is easy to see that $\Sigma^1_{n+1}\hyp\DC_0$ proves that $\varphi(x)$ is inductive:
$\varphi(1) \land \forall x(\varphi(x) \to \varphi(x+1))$.
Then, by Solovay's shortening cut method, there is a polynomial $f$ such that 
$\Sigma^1_{n+1}\hyp\DC_0$ has a proof of $\varphi(2_k)$ with size\footnote{Here, $2_k$ is expressed by using a function symbol corresponding to (or a formula defining) the super-exponential function. } $f(k)$ for all $k > 0$.
Here, $2_k$ denotes the super-exponential function; $2_0 = 1$ and $2_{k+1} = 2^{2_k}$. 
For the details, see \cite[Theorem~3.4.1]{Pudlak}. See also \cite[the paragraph following Theorem~3.9]{suzuki2025pi}.
On the other hand, if $\ACAo + \beta_n\RFN^{\omega}(\ACAo)$ is defined as $\ACAo + \{\varphi(\overline{r}) : r > 0\}$, then it
has only proofs of $\varphi(2_k)$ whose size is greater than $2_k$. This is because 
$\ACAo + \beta_n\RFN^r(\ACAo) \vdash \Con(\ACAo + \beta_n\RFN^{r'}(\ACAo))$ for all $0 < r' < r$. Hence, to prove $ \varphi(2_k)$ within $\ACAo + \beta_n\RFN^{\omega}(\ACAo)$, we need to use 
$\varphi(\overline{r})$ for some $r \geq 2_k$.
Clearly, the size of such a proof is greater than $2_k$.
By the same argument as above, we obtain the super-exponential speed-up for cases
(2), (4), (6)--(8), (10), (12), (13), (15).

We next consider case~(3).
We prove that every proof of $\beta_n\RFN^{2_k}(\ACAo)$ over $\ACAo + (\Pi^1_{n+2})^{-}\hyp\beta_n\RFN$ has size greater than 
$2_k$. As we have proved that $\Sigma^1_{n+1}\hyp\DC_0$ has a proof with $\mathrm{poly}(k)$ size, this shows the super-exponential speed-up.

\begin{lemma}\label{lem:local-reflection-finite-stage}
  Let $\sigma_0,\ldots, \sigma_{r -1}$ and $\pi$ be $\Pi^1_{n+2}$-sentences.
  If $\ACAo + \{\sigma_i \to \beta_n\RFN(\sigma_i) : i < r\} \vdash \pi$, then 
  $\ACAo + \beta_n\RFN^{r}(\ACAo) \vdash \pi$.
\end{lemma}
\begin{proof}
   For each subset $I \subseteq \{0,\ldots,r-1\}$, put $\delta_I = \pi \lor \bigvee_{i \in I} \sigma_i$.
   We prove $\ACAo + \beta_n\RFN^{r -|I|}(\ACAo) \vdash \delta_I$ for all $I \subseteq \{0,\ldots,r-1\}$ by induction on $r - |I|$. Here, we regard $\beta_n\RFN^{0}(\ACAo)$ as $\top$.

   Assume $r - |I| = 0$, that is, $I = \{0,\ldots,r-1\}$.
   We work in a model of $\ACAo$. If $\lnot \bigvee_{i \in I} \sigma_i$ holds in this model, then $\bigwedge_{i \in I} (\sigma_i \to \beta_n\RFN(\sigma_i))$ holds. Then, this model satisfies $\ACAo + \{\sigma_i \to \beta_n\RFN(\sigma_i) : i < r\}$. Thus, $\pi$ holds in this model.

   We next consider the induction step.
   Let $I \subseteq \{0,\ldots,r-1\}$. By the induction hypothesis, we may assume that 
   $$ 
   \ACAo + \beta_n\RFN^{r - |I| -1}(\ACAo) \vdash \delta_{I} \lor \sigma_j \text{ for all } j \in I^c.
   $$
   We work in (a model of) $\ACAo + \beta_n\RFN^{r - |I|}(\ACAo)$. We will show that $\delta_I$ holds.
   By $\beta_n\RFN^{r- |I|}$, we obtain 
   $\forall X \exists \M(X \in \M \land \beta_n(\M) \land \M \models \beta_n\RFN^{r - |I| -1}(\ACAo) )$.
   Thus, by the assumption, 
   $$
   \forall X \exists \M(X \in \M \land \beta_n(\M) \land \M \models \delta_{I} \lor \sigma_j) \text{ for all } j \in I^c.
   $$
   Thus, we obtain $\beta_n\RFN(\delta_I \lor \sigma_j)$ for all $j \in I^c$.
   By Lemma~\ref{lem BP can be applied to beta_nRFN}, $\delta_I \lor \beta_n\RFN(\sigma_j)$ holds for all $j \in I^c$.
   If $\delta_I$ failed, then we would obtain 
   $\lnot \pi \land \bigwedge_{i \in I} \lnot \sigma_i \land \bigwedge_{j \in I^c} \beta_n\RFN(\sigma_j)$.
   In particular, $\bigwedge_{i < r} (\sigma_i \to \beta_n\RFN(\sigma_i))$ would hold.
   Thus, $\pi$ would hold, a contradiction. Thus, $\delta_I$ holds as desired.

   Now we have $\ACAo + \beta_n\RFN^{r - |\varnothing|}(\ACAo) \vdash \delta_{\varnothing}$.
   This means that $\ACAo + \beta_n\RFN^{r}(\ACAo) \vdash \pi$.
\end{proof}
\begin{theorem}
  Every proof of $\beta_n\RFN^r(\ACAo)$ over $(\Pi^1_{n+2})^{-}\hyp\beta_n\RFN_0$ uses at least $r$ instances of $(\Pi^1_{n+2})^{-}\hyp\beta_n\RFN$. In particular, such a proof has size greater than $r$.
\end{theorem}
\begin{proof}
  Let $\sigma_0,\ldots, \sigma_{l -1}$ be $\Pi^1_{n+2}$-sentences such that 
  $\ACAo + \{\sigma_i \to \beta_n\RFN(\sigma_i) : i < l\} \vdash \beta_n\RFN^r(\ACAo)$.
  Then, $\ACAo + \beta_n\RFN^l(\ACAo) \vdash \beta_n\RFN^r(\ACAo)$. Thus, $l \geq r$ as discussed in the paragraph preceding Lemma~\ref{lem:local-reflection-finite-stage}.
\end{proof}
As we have mentioned in the paragraph before Lemma~\ref{lem:local-reflection-finite-stage}, this implies the super-exponential speed-up.
By the same argument as above, we obtain the super-exponential speed-up for cases
(9) and (14).

We finally consider case~(5).
Let $\psi(x)$ be a $\Sigma^1_{m+2}$-formula such that 
$\ACAo \vdash \psi(\overline{r}) \leftrightarrow \exists \M(\beta_m(\M) \land \M \models \beta_n\RFN^{r}(\ACAo))$ for $r \geq 0$. Here, we regard $\beta_n\RFN^0(\ACAo)$ as $\top$. 
Then, $\Sigma^1_{n+1}\hyp\DC_0$ has a poly($k$) size proof of $\psi(2_k)$ as in case~(1).
We will show that every proof of $\psi(2_k)$ over
$\ACAo + \{\sigma \to \exists \M(\beta_m(\M) \land \M \models \beta_n\RFN^r(\sigma)) : \sigma \in (\Pi^1_{m+2})^{-}, r > 0\}$ has size greater than $2_k$.
\begin{lemma}
  Let $m \leq n$ and $r > 0$. 
  Let $\sigma$ be a $\Pi^1_{m+2}$-sentence.
  Then,
  there exist $\Pi^1_{n+2}$-sentences
  $\tau_0,\ldots,\tau_{r}$ such that 
  $\ACAo + \{\tau_i \to \beta_n\RFN(\tau_i) : i \leq r \} \vdash  \sigma \to \exists \M(\beta_m(\M) \land \M \models \beta_n\RFN^r(\sigma))$.
\end{lemma}
\begin{proof}
  Let $m \leq n$ and $r > 0$.
  Put $\tau_0 \equiv \sigma$ and $\tau_{i+1} \equiv \sigma \land \beta_n\RFN^{i+1}(\sigma)$ for each $i+1 \leq r$, and 
  $T = \ACAo + \{\tau_i \to \beta_n\RFN(\tau_i) : i \leq r\}$.
  Then, $T \vdash \sigma \to \beta_n\RFN^{r+1}(\sigma)$.
  In particular, $T \vdash \sigma \to \exists \M (\beta_n(\M) \land \M \models \beta_n\RFN^r(\sigma))$, and hence 
  $T \vdash \sigma \to \exists \M(\beta_m(\M) \land \M \models \beta_n\RFN^r(\sigma))$.
\end{proof}

\begin{theorem}
   Every proof of $\exists \M(\beta_m(\M) \land \M \models \beta_n\RFN^r(\ACAo))$ over $\ACAo + \{\sigma \to \exists \M(\beta_m(\M) \land \M \models \beta_n\RFN^s(\sigma)) : \sigma \in (\Pi^1_{m+2})^{-}, s > 0 \}$ has size greater than $r$.
\end{theorem}
\begin{proof}
  We fix a proof of $\exists \M(\beta_m(\M) \land \M \models \beta_n\RFN^r(\ACAo))$ over $\ACAo + \{\sigma \to \exists \M(\beta_m(\M) \land \M \models \beta_n\RFN^s(\sigma)) : \sigma \in (\Pi^1_{m+2})^{-}, s > 0 \}$.
  Let $\sigma_0,\ldots,\sigma_{I-1} \in (\Pi^1_{m+2})^{-}$ be such that the proof contains 
  $\sigma_i \to \exists \M(\beta_m(\M) \land \M \models \beta_n\RFN^{r_i}(\sigma_i))$ for all $i < I$.
  Then, for each $i < I$, there exist $\tau_0^{i},\ldots,\tau_{r_i}^i \in (\Pi^1_{n+2})^{-}$ such that 
  $$
  \ACAo + \{\tau^i_j \to \beta_n\RFN(\tau^i_j) : j \leq r_i \} \vdash \sigma_i \to \exists \M(\beta_m(\M) \land \M \models \beta_n\RFN^{r_i}(\sigma_i)).
  $$
  Then, 
  $$
  \ACAo + \{\tau^i_j \to \beta_n\RFN(\tau^i_j) : j \leq r_i, i < I \} \vdash \exists \M(\beta_m(\M) \land \M \models \beta_n\RFN^{r}(\ACAo)).
  $$
  Since $\exists \M(\beta_m(\M) \land \M \models \beta_n\RFN^{r}(\ACAo))$ is $\Sigma^1_{m+1}$, it is $\Pi^1_{m+2}$.
  Hence, by Lemma~\ref{lem:local-reflection-finite-stage}, $\ACAo + \beta_n\RFN^{s}(\ACAo) \vdash \exists \M(\beta_m(\M) \land \M \models \beta_n\RFN^{r}(\ACAo))$ for $s = \sum_{i < I} (r_i+1)$.
  In particular, $\ACAo + \beta_n\RFN^s(\ACAo) \vdash \Con(\ACAo + \beta_n\RFN^r(\ACAo))$.
  Thus, $s > r$.
  Since the size of the proof is greater than $s$, this completes the proof.
\end{proof}

\bibliographystyle{plain}
\bibliography{references}

\end{document}